\documentclass[final]{siamltex}

\usepackage{latexsym, enumerate}
\usepackage{eepic}
\usepackage{epic}
\usepackage{graphicx}
\usepackage{color}
\usepackage{ifpdf}
\usepackage{amssymb,amsmath,epsfig, algorithm,algorithmicx,multirow}
\usepackage{amsfonts, dsfont}
\usepackage{subfigure}

\newcommand {\aux}{\text{aux}}
\newcommand {\f}{\text{frac}}

\title{A Constraint energy minimizing generalized multiscale finite element method for parabolic equations
\thanks{Eric Chung's work is partially supported by Hong Kong RGC General Research Fund (Project 14304217)
and CUHK Direct Grant for Research 2017-18. L. Jiang acknowledges  the support of Chinese NSF 11471107.
}}

\author{Mengnan Li\thanks{College of Mathematics and Econometrics, Hunan University, Changsha 410082, China ({\tt mengnanli@hnu.edu.cn}).}
\and
Eric Chung\thanks{Department of Mathematics, The Chinese University of Hong Kong, Hong Kong SAR. ({\tt tschung@math.cuhk.edu.hk}).}
\and
Lijian Jiang\thanks{School of Mathematical Sciences,  Tongji University, Shanghai 200092, China. ({\tt  ljjiang@tongji.edu.cn}).Corresponding author}
}

\begin{document}

\maketitle

\begin{abstract}
In this paper, we present a Constraint Energy Minimizing Generalized Multiscale Finite Element Method (CEM-GMsFEM) for parabolic equations with multiscale coefficients, arising from applications in  porous media. We will present the construction of  CEM-GMsFEM  and rigorously analyze its convergence for  the parabolic equations.
The convergence rate is characterized  by the coarse grid size and the eigenvalue decay of local spectral problems, but is independent of the scale length  and contrast of the media.
The analysis shows  that the method has a first order convergence rate with respect to coarse grid size in the energy norm and second order convergence rate with respect to coarse grid size in  $L^2$ norm
under some appropriate assumptions. For the temporal discretization, finite difference techniques are used  and the convergence analysis of full discrete scheme is given.
Moreover, a posteriori error estimator  is derived and analyzed. A few  numerical results for porous media applications  are presented   to
confirm the theoretical findings and demonstrate the performance of the approach.

\end{abstract}



\pagestyle{myheadings}

\thispagestyle{plain}
\markboth{M. Li,  E. Chung and L. Jiang}{CEM-GMsFEM for parabolic equations}

\section{Introduction}
The multiscale characteristic of formation properties pose significant challenges for subsurface flow modeling. Geological characterizations that capture these effects are typically developed at scales that are too fine for direct flow simulation, so reduced models are necessary to compute the solution of flow problems in practice. There are a number of methods that have been developed to solve multiscale problems. The upscaling method \cite{upsca2,upsca4} is one of the classical approaches, which is based on the homogenization theory (e.g., \cite{homo1,homo2,homo3,para1}). The key to upsacling techniques is to form a coarse-scale equation using pre-computed effective coefficients \cite{weh02}. The multiscale finite element method (MsFEM) \cite{hw97} share similarities with upscaling methods, but the coarse-scale equations are obtained through variational formulation  by multiscale basis functions, which contain fine scale information. Some other multiscale methods such as variational multiscale methods \cite{Hughes98, ar02} and multi-scale finite volume method \cite{fvm1} share the model reduction techniques but using different frameworks to solve multiscale problems.

Recently a multiscale method \cite{Eric2017CEM} for which the basis functions are constructed by the principle of constraint energy minimization is proposed. The approach shares some ideas of the generalized multiscale finite element method (GMsFEM) \cite{egh13,ccj16, GMs3}, which is a systematic way to find the multiscale basis functions. In particular, the GMsFEM constructs a suitable local snapshot space and then solve a spectral problem in each coarse block. The basis functions are the dominant eigenvectors corresponding to small eigenvalues. The GMsFEM's convergence depends on the eigenvalue decay of the local spectral problems \cite{egh13}. To obtain a convergence depending on the coarse  mesh, a new methodology, called constraint energy minimizing GMsFEM (CEM-GMsFEM),
is proposed in \cite{Eric2017CEM}. The idea of CEM-GMsFEM can be divided in two steps. First, this method is to construct auxiliary basis functions via local spectral problems. Then,  constraint energy minimization problems are  solved to obtain the required basis functions, where the constraints are related to the auxiliary basis functions. Combining the eigenfunctions and the energy minimizing properties, the basis functions are shown to have exponential decay away from the target coarse block, and thus the basis functions can be computed locally. In addition, the convergence depends on the coarse mesh size when the oversampling domain is  carefully chosen. Also, the size of oversampling domain depends weakly on the contrast of the media.

 There are some multiscale methods to solve the parabolic equations with multiscale coefficients \cite{para2,para1, finitv1,oz07}. These methods can get accurate solutions in some cases, but one may need to use multiple  multiscale basis functions which can be computed locally in order to accurately capture complex multiscale features. The aim of this paper is to solve these problems by  CEM-GMsFEM and carry out convergence analysis for parabolic equations with high-contrast coefficients. The multiscale basis functions allow a spatial decay and can be used in models with the similar multiscale coefficients and different source terms, boundary conditions and initial conditions. The convergence analysis of the semidiscrete formulation is made first. We use an  elliptic spectral  projection as the  bridge between the multiscale finite element space and the continuous space. Then, we  show that the method is first order convergence rate  with respect to coarse grid size  in the energy norm, and second order convergence rate with respect coarse grid size  in the $L^2$ norm, under some mild assumptions on the solutions of parabolic equations. In addition, we also consider the fully discrete scheme, and prove its stability and convergence. Furthermore, we  derive an a posteriori error bound for the CEM-GMsFEM scheme, which can be served
as a foundation for adaptive enrichment schemes \cite{posterior3,posterior2,posterior4,posterior5,posterior1}. We prove that the error between the multiscale finite element  solution and the fine-scale solution is bounded by a local error indicator.

   We consider two versions of  CEM-GMsFEM. The first one is based on the constraint energy minimization and the second one is based on the relax version by solving unconstrainted energy minimization. We present the numerical results using the heterogeneous permeability fields with channels and inclusions. By selecting the number of basis functions and oversampling layers properly, the numerical results can verify our theoretical estimates and convergence rate. In subsurface  problems, the model with a fracture structure  is of importance. The permeability in  fractures is very different from that in the surrounding matrix.
   And the discrete fracture model (DFM)  is one of the important fracture models and has  been extensively  studied  \cite{frac1,frac3,frac2}. In the numerical result section, we apply the CEM-GMsFEM  to DFM in a single phase flow.  The numerical results of a posteriori error bound is also given and show the robustness of the  a posteriori error estimator.

   This paper is organized as follows. In Section~2, we will give some preliminaries. In Section~3, basis functions of  CEM-GMsFEM are  presented. In Section~4, we make   the convergence analysis  of the semi-discrete formulation. In Section~5, the analysis of the  fully discrete scheme is presented. In Section~6, a few numerical results are presented to confirm the convergence  analysis. Finally, some conclusions are  given.

\section{Preliminaries}

Let $\Omega$ be a bounded computational domain in $R^{d}$ ($d=2,3$) and $T>0$ be a fixed time. We consider the following model parabolic equation
\begin{equation}\label{ex-eq}
  \left\{
 \begin{aligned}
 u_{t} - \nabla \cdot \big(\kappa(x) \nabla u\big)&=f(x,t) \quad in\quad \Omega\times (0,T], \\
 u&=0 \quad on\quad \partial\Omega\times (0,T],\\
 u(\cdot,0)&=u_{0}(x) \quad in\quad \Omega.
 \end{aligned}
   \right.
 \end{equation}
  We assume that multiscale coefficient $\kappa(x)$ is uniformly positive and bound in $\Omega$, i.e., there exists $0< \kappa_0 < \kappa_1<\infty$ such that $\kappa_0\leq\kappa(x)\leq \kappa_1$, but the ratio $\kappa_1/\kappa_0$ can be large.  This parabolic model is often used in the subsurface flow modeling. We give  the notions of fine grid and coarse grid for the domain $\Omega$. Let $\mathrm{H}$ be the coarse-mesh size and  $\mathcal{T}^{H}$ a finite element conforming coarse partition of the domain $\Omega$.
  Each coarse grid block in $\mathcal{T}^{H}$ is a connected union of fine grid blocks in the fine partition  $\tau^{h}$, where $h$ denotes the fine grid size. Let $N_{c}$ be the number of coarse vertices and $N$
   the number of elements of $\mathcal{T}^{H}$. Let  $\{x_{i}\}_{i=1}^{N_{c}}$ be  the set of  vertices in $\mathcal{T}^{H}$ and $\omega_{i}=\bigcup\{K_{j}\in \mathcal{T}^{H}|x_{i}\in\overline{K_{j}}\}$
    the neighborhood of the node $x_{i}$.  Figure \ref{grid1} illustrates  the fine grid, coarse grid $K_i$, the neighborhood $\omega_{i}$ and  oversampling domain $K_{i,1}$ obtained using one coarse element layer extension.
In this paper, we use the usual  notations (e.g., $H^1(\Omega)$) for standard  Sobolev spaces.  We also use the conventional   notations  such as $L^2(0, T; H^1(\Omega)$ for space-time dependent Sobolev spaces.
For simplicity of notations, we will suppress the time variable $t$ and the space variable $x$ in functions  when no ambiguity occurs.

\begin{figure}
  \centering
  \includegraphics[width=6in]{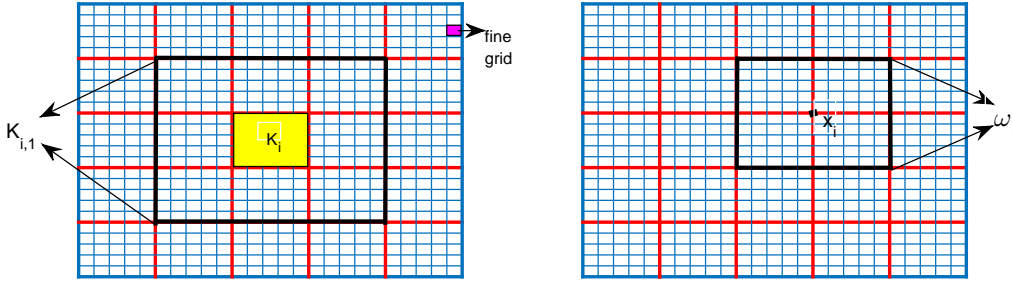}
  \caption{The fine grid, coarse grid $K_{i}$, oversampling domain $K_{i,1}$ and the neighborhood $\omega_{i}$ of the node $x_{i}$}\label{grid1}
\end{figure}

Let $V:=H^1(\Omega)$ and  $V_{0}:=H^1_{0}(\Omega)$.
The weak formulation  of (\ref{ex-eq}) is to  find $u(\cdot,t) \in V_{0}$ such that
\begin{equation}\label{eq-weak}
\left\{
 \begin{aligned}
(u_t,v) + a(u,v)& = (f,v) \quad \forall v\in V_{0}, \, t> 0\\
\big(u(\cdot,0), v\big)&=(u_{0},v) \quad \forall v\in V_{0},
\end{aligned}
   \right.
\end{equation}
where
\[
a(u,v) = \int_{\Omega} \kappa \nabla u \cdot \nabla v, \quad (f,v)=\int_{\Omega}fv.
\]
The aim of this paper is to construct a multiscale method for (\ref{ex-eq}). To this end, we will construct a finite dimensional  multiscale space $V_{ms} \subset V_{0}$,
whose dimension is small. Then we will find the multiscale solution $u(\cdot,t)\in V_{ms}$ by solving
\begin{equation}\label{eq:scheme}
\left\{
 \begin{aligned}
\big((u_{ms})_t,v\big) + a(u_{ms},v)& = (f,v) \quad \forall v\in V_{ms},\, t> 0\\
\big(u_{ms}(\cdot,0), v\big)&=(u_{0},v) \quad \forall v\in V_{ms}.
\end{aligned}
   \right.
\end{equation}
We denote $u_{0,ms}\in V_{ms}$ as the $L^2$ projection of $u_0$ onto  $V_{ms}$.

\section{Construction of CEM-GMsFEM basis functions}

In this section, we briefly describe the construction of CEM-GMsFEM basis functions. This process can be divided into two stages. The first stage is the construction of the auxiliary space by solving  a local spectral problem in each coarse element $K$. The second stage   is to construct the multiscale basis functions by solving some local constraint energy minimization problems in oversampling domains.

\subsection{Auxiliary basis function}
 In this subsection, we present how to construct the auxiliary space. Let  $V(K_{i})=H^{1}(K_{i})$ for a generic coarse block $K_i$. We  solve the following spectral problem in each coarse element $K_{i}$: find
 eigen-pairs $\{\lambda_{j}^{(i)},  \phi_{j}^{(i)}\}$  such that
\[
a_{i}(\phi_{j}^{(i)},v) =\lambda_{j}^{(i)}s_{i}(\phi_{j}^{(i)},v)\quad \forall v\in V(K_{i}),
\]
where
\[
a_{i}(u,v) = \int_{K_{i}} \kappa \nabla u \cdot \nabla v, \quad s_{i}(u,v)=\int_{K_{i}}\widetilde{\kappa}uv.
\]
 Here $\widetilde{\kappa}=\kappa\sum_{j=1}^{N_{c}}|\nabla\chi_{j}|^{2}$ and $\{\chi_{j}\}$ is a set of partition of unity functions for the coarse partition $\mathcal{T}^{H}$. We assume that the eigenfunctions satisfy the normalized condition $s_{i}(\phi_{j}^{(i)},\phi_{j}^{(i)})=1$. Let the eigenvalues $\lambda_{j}^{(i)}$ be arranged in ascending order, i.e., $\lambda_{1}^{(i)}\leq \lambda_{2}^{(i)}\leq \cdots$. We define the local auxiliary multiscale space $V_{aux}^{(i)}$ by using the first $L_{i}$ eigenfunctions
 \[
 V_{\aux}^{(i)}=\text{span} \{\phi_{j}^{(i)}|1\leq j\leq L_{i}\}.
 \]
 Then the global auxiliary space $V_{\aux}$ is defined by using these local auxiliary spaces, i.e.,
 \[
 V_{\aux}=\bigoplus_{i=1}^{N}V_{\aux}^{(i)}.
\]
To construct the CEM-GMsFEM basis functions, we need the following definition.
\begin{definition}\label{defn1}
($\phi_{j}^{(i)}$-orthogonality) Given a function $\phi_{j}^{(i)}\in V_{\aux}$, if a function $\psi\in V$ satisfies
\[
s(\psi,\phi_{j}^{(i)}) = 1,\quad s(\psi,\phi_{j'}^{(i')}) = 0\quad if \ j'\neq j \ or \  i'\neq i,
\]
then we say that  $\psi$ is $\phi_{j}^{(i)}$-orthogonal. Where $s(u,v)=\sum_{i=1}^{N}s_{i}(u,v)$.
\end{definition}

We define an  operator $\pi:V\rightarrow V_{\aux}$ by
\[
\pi(v) = \sum_{i=1}^{N}\sum_{j=1}^{L_{i}}s_{i}(v,\phi_{j}^{(i)})\phi_{j}^{(i)},\quad\forall v\in V.
\]
The null space of the operator $\pi$ is defined by  $\widetilde{V}=\{v\in V|\pi(v)=0\}$.
\subsection{Multiscale basis functions in CEM-GMsFEM}

 Now we present the construction of the multiscale basis functions. Given a coarse block  $K_i$, we denote the oversampling  region $K_{i,m}\subset \Omega$ obtained by enlarging  $K_{i}$ with $m$ coarse grid layers (see Figure \ref{grid1}). Let $V_{0}(K_{i,m}):=H_0^1(K_{i,m})$.   Then, we define the multiscale basis function $\psi_{j,ms}^{(i)}\in V_{0}(K_{i,m})$ by
\begin{equation}\label{min}
\psi_{j,ms}^{(i)}=\text{argmin}\big\{a(\psi,\psi)|\psi\in V_{0}(K_{i,m}), \psi \ is \ \phi_{j}^{(i)}\text{-orthogonal}\big\}.
\end{equation}
Then the CEM-GMsFEM space is defined by
\[
V_{ms}=\text{span}\big\{\psi_{j,ms}^{(i)}|1\leq j\leq L_{i},1\leq i\leq N\big\}.
\]
By using Lagrange Multiplier, the problem (\ref{min}) can be rewritten  as the following problem:
find $\psi_{j,ms}^{(i)}\in V_{0}(K_{i,m}),\quad \lambda \in V_{\aux}^{(i)}$ such that
\[
\left\{
 \begin{aligned}
 a(\psi_{j,ms}^{(i)},\omega)+s(\omega,\lambda)& = 0 \quad \forall \omega\in V_{0}(K_{i,m}), \\
s(\psi_{j,ms}^{(i)}-\phi_{j}^{(i)},\nu)&=0  \quad \forall \nu\in V_{\aux}^{(i)}(K_{i,m}),
\end{aligned}
   \right.
\]
where $V_{\aux}^{(i)}(K_{i,m})$ is the union of all local auxiliary spaces for $K_j \subset K_{i,m}$.

 By \cite{Eric2017CEM}, we can relax the $\phi_{j}^{(i)}$-orthogonality in (\ref{min}) and get  a relaxed CEM-GMsFEM. This can be achieved by solving the unconstrained minimization problem: find $\psi_{j,rms}^{(i)}\in V_{0}(K_{i,m})$ such that
 \begin{equation}\label{uncon1}
   \psi_{j,rms}^{(i)}=\text{argmin}\big\{a(\psi,\psi)+s(\pi\psi-\phi_{j}^{(i)},\pi\psi-\phi_{j}^{(i)})|\ \psi\in V_{0}(K_{i,m})\big\}.
 \end{equation}
 Then  the resulting  multiscale finite element  space is defined by
 \[
 V_{rms}=\big\{\psi_{j,rms}^{(i)}|1\leq j\leq L_{i}, 1\leq i\leq N\big\}.
 \]
 We note that (\ref{uncon1}) is equivalent to
 \[
 a(\psi_{j,rms}^{(i)},\omega)+s(\pi\psi_{j,rms}^{(i)}-\phi_{j}^{(i)},\pi\omega)=0,\ \forall \ \omega\in V_{0}(K_{i,m}).
 \]
\section{The Semidiscrete formulation}
In this section, we approximate the solution $u(x,t)$ of the parabolic problem (\ref{eq-weak}) by a function $u_{ms}$  which belongs to the space $V_{ms}$ for each $t>0$.
 The solution $u_{ms}$ solves equation (\ref{eq:scheme}).     We define the following  norms for our  analysis:
\[
\|u\|_{a}^{2}:=\int_{\Omega}\kappa|\nabla u|^{2}, \quad
\|u\|_{s}^{2}:=\int_{\Omega}\widetilde{\kappa}|u|^{2}.
\]
\subsection{Error estimates}
 In this subsection, we will give  the error estimates of the proposed method.
 To this end,  we need  some preliminary lemmas.

Let $\widehat{u}\in V_{ms}$ be the elliptic projection of the function $u\in V$, i.e., $\widehat{u}$ satisfies
\[
a(u-\widehat{u},v)=0 \ \forall v\in V_{ms}.
\]
The following lemma gives an estimate between the solution of elliptic equation and its elliptic projection.
\begin{lemma}\label{lem12} \cite{Eric2017CEM}. Let $u$ be the solution of the elliptic equation
 \begin{equation}\label{elliptic}
 \left\{
 \begin{aligned}
-div(\kappa\nabla u)&=f \quad \text{in} \; \Omega\\
u&=0 \quad \text{on} \;\partial\Omega,
\end{aligned}
   \right.
\end{equation}
and $\widehat{u}\in V_{ms}$ be the elliptic projection of the $u$ in the subspace of $V_{ms}$.
If $\{\chi_{j}\}$ are the bilinear partition of unity functions and $m=O(\log(\frac{\kappa_1}{H\kappa_0}))$, where $m$ is the number of oversampling layers in the construction of basis in (\ref{min}). Then
\begin{equation}\label{elli1}
  \|u-\widehat{u}\|_{a}\leq CH\Lambda^{-\frac{1}{2}}\|\kappa^{-\frac{1}{2}}f\|,
\end{equation}
where $\Lambda=\min_{1\leq i\leq N}\lambda_{L_{i}+1}^{(i)}$.
\end{lemma}

The following lemma gives a regularity estimate for the solution of the parabolic equation (\ref{ex-eq}).
\begin{lemma}\label{lem14}
 Let $u$ be the solution of the parabolic equation (\ref{ex-eq}). Then
\begin{equation}\label{ut}
  \|u_{t}\|^{2}_{L^2(0,T; L^2(\Omega))}\leq C\bigg(\|u_{0}\|_a^{2}+\|f\|^{2}_{L^2(0,T; L^2(\Omega))} \bigg).
\end{equation}
\begin{proof}
By (\ref{ex-eq}), we have
\[
(u_{t},u_{t}) + a(u,u_t) =(f,u_{t}),
\]
which implies
\[
\|u_t\|^2 + \frac{1}{2} \frac{d}{dt} \|u\|_a^2 = (f,u_t).
\]
Integrating with respect to  time, we have
\[
\int_0^T \|u_t\|^2 + \frac{1}{2} \|u\|_a^2 = \int_0^T (f,u_t) + \frac{1}{2} \|u_0\|_a^2.
\]
This completes the proof.
\end{proof}
\end{lemma}

 To estimate  the error bound, the elliptic projection $\widehat{u}\in V_{ms}$ of the  solution $u$ plays an important role. The following lemma
gives the error estimate of $\widehat{u}(t)$ for the parabolic equation.
\begin{lemma}
\label{lem1}
 Let $u$ be the solution of (\ref{eq-weak}). For each $t>0$ , we define the elliptic projection $\widehat{u}(t)\in V_{ms}$ by
\begin{equation}\label{ms-ell}
a\big(u(t)-\widehat{u}(t),v\big) =0, \quad \forall v\in V_{ms}.
\end{equation}
 Then, for any $t>0$,
\begin{equation}\label{ms-a-err}
\| (u-\widehat{u})(t)\|_{a} \leq CH\Lambda^{-\frac{1}{2}}\|\kappa^{-\frac{1}{2}}(f-u_{t})(t)\|,
\end{equation}
\begin{equation}\label{ms-l2-err}
\|(u-\widehat{u})(t)\|\leq CH^{2}\Lambda^{-1} \kappa_0^{-\frac{1}{2}} \|\kappa^{-\frac{1}{2}}(f-u_{t})(t)\|,
\end{equation}
where $C$ is a constant independent of $\kappa$ and the mesh size $H$.
\end{lemma}
\begin{proof}
Note that the solution $u\in V_{0}$ of (\ref{eq-weak}) satisfies
\[
a(u,v) = (f-u_{t},v), \quad \forall v\in V_{0},\ \forall t> 0.
\]
Thus, $\widehat{u}(t)\in V_{ms}$ satisfies
\[
a(\widehat{u},v) =a(u,v)= (f-u_{t},v), \quad \forall v\in V_{ms},\ \forall t> 0.
\]
By Lemma \ref{lem12}, we obtain
\[
\begin{split}
\| u-\widehat{u}\|_{a}&\leq CH\Lambda^{-\frac{1}{2}}\|\kappa^{-\frac{1}{2}}(f-u_{t})\|, \quad\forall t>0.
\end{split}
\]
Next, we derive the error estimate in the $L^{2}$ norm.
We will apply the Aubin-Nitsche lift technique.
For each $t>0$, we define $w\in V_{0}$ by
\[
a(w,v) = (u-\widehat{u}, v), \quad \forall v\in V_{0},
\]
and define $\widehat{w}$ as the elliptic projection of $w$ in the space $V_{ms}$, that is,
\[
a(\widehat{w},v) = (u-\widehat{u}, v), \quad \forall v\in V_{ms}.
\]
By Lemma \ref{lem12}, we obtain
\[
\begin{split}
\| u-\widehat{u}\|^2 & = a(w,u-\widehat{u})\\
&= a(w-\widehat{w},u-\widehat{u})\\
&\leq \|w-\widehat{w}\|_{a} \, \|u-\widehat{u}\|_{a}\\
& \leq C( H\Lambda^{-\frac{1}{2}}\max\{\kappa^{-\frac{1}{2}}\}\|u-\widehat{u}\| ) \, \big(H \Lambda^{-\frac{1}{2}} \|\kappa^{-\frac{1}{2}}(f-u_{t})\| \big)
\end{split}
\]
where the second equality is obtained by (\ref{ms-ell}).
Hence, we have
\[
\| u-\widehat{u}\| \leq CH^{2}\Lambda^{-1} \kappa_0^{-\frac{1}{2}}\|\kappa^{-\frac{1}{2}}(f-u_{t})\|, \quad\forall t>0.
\]
This completes the proof.
\end{proof}

In the following theorems, we will prove the error estimates in the  $L^{2}$ norm and energy norm.
We will first state and prove the estimate in energy norm.

\begin{theorem}
\label{thm2}
Let $u$ and $u_{ms}$ be the solution of (\ref{eq-weak}) and (\ref{eq:scheme}),  respectively. Then
\begin{equation}
\begin{split}
&\: \| (u - u_{ms})(\cdot,T) \|^{2} +\int_0^T \| u - u_{ms} \|_{a}^{2} \\
\leq &\: CH^{2}\Lambda^{-1} \kappa_0^{-\frac{1}{2}} \|\kappa^{-\frac{1}{2}}(f-u_{t})\|_{L^2(0,T;L^2(\Omega))} \bigg(\|u_{0}\|_a^{2}+\|u_{0,ms}\|_a^{2} + \|f\|^{2}_{L^2(0,T;L^2(\Omega))} \bigg)^{\frac{1}{2}} \\
\quad &\: + CH^2\Lambda^{-1}\|\kappa^{-\frac{1}{2}}(f-u_{t})\|_{L^2(0,T;L^2(\Omega))}^2
\end{split}
\end{equation}
where $C$ is a constant independent of $\kappa$ and the mesh size $H$.
\end{theorem}
\begin{proof}
By  (\ref{eq-weak}) and (\ref{eq:scheme}), we have
\[
((u-u_{ms})_t,v) + a(u-u_{ms},v) = 0
\]
for all $v\in V_{ms}$.
We define $\widehat{u}\in V_{ms}$ as the elliptic projection of $u$, which satisfies (\ref{ms-a-err}) and (\ref{ms-l2-err}). Taking $v=\widehat{u}-u_{ms}$,
we have the following  equation
\[
((u-u_{ms})_t,\widehat{u}-u_{ms}) + a(u-u_{ms},\widehat{u}-u_{ms}) = 0.
\]
This can be written as
\[
((u-u_{ms})_t,u-u_{ms}) + a(u-u_{ms},u-u_{ms}) =
((u-u_{ms})_t,u-\widehat{u}) + a(u-u_{ms},u-\widehat{u}).
\]
This implies 
\[
\frac{1}{2} \frac{d}{dt} \| u - u_{ms} \|^2 + \| u - u_{ms} \|_{a}^2
=  ((u-u_{ms})_t,u-\widehat{u}) + a(u-u_{ms},u-\widehat{u}).
\]
By  Cauchy-Schwarz inequality and Young's inequality, we have
\[
\begin{split}
\frac{1}{2} \frac{d}{dt} \| u - u_{ms} \|^2 + \| u - u_{ms} \|_{a}^2
&\leq  \|u_t-(u_{ms})_t\|\| u - \widehat{u}\|  + \| u-u_{ms}\|_{a} \|u-\widehat{u}\|_{a}\\
&\leq  ( \|u_t\|+\|(u_{ms})_t\|) \| u - \widehat{u}\|  +\frac{1}{2} \| u-u_{ms}\|_{a}^{2}+\frac{1}{2} \|u-\widehat{u}\|_{a}^{2}.
\end{split}
\]
So, we have
\[
\frac{d}{dt} \| u - u_{ms} \|^2 + \| u - u_{ms} \|_{a}^2\leq  2( \|u_t\|+\|(u_{ms})_t\|) \| u - \widehat{u}\|+\|u-\widehat{u}\|_{a}^{2}.
\]
Integrating with respect to  time, we have
\begin{equation*}
\begin{split}
&\:  \| (u - u_{ms})(\cdot,T) \|^2 -  \| (u - u_{ms})(\cdot,0) \|^2 + \int_0^T \| u - u_{ms} \|_{a}^2 \\
 \leq &\:  \int_0^T 2( \|u_t\|+\|(u_{ms})_t\|) \| u - \widehat{u}\|+ \int_0^T \|u-\widehat{u}\|_{a}^{2}.
\end{split}
\end{equation*}
By (\ref{ut}), we can bound $u_t$ by the initial data and source function $f$, that is
\begin{equation*}
 \int_0^T  \|u_{t}\|^{2}\leq C\bigg(\|u_{0}\|_a^{2}+\int_{0}^{T}\|f\|^{2} \bigg).
\end{equation*}
Similarly, by (\ref{eq:scheme}), we have
\begin{equation*}
 \int_0^T  \|(u_{ms})_t\|^{2}\leq C\bigg(\|u_{0,ms}\|_a^{2}+\int_{0}^{T}\|f\|^{2} \bigg).
\end{equation*}

Thus, (\ref{ms-a-err}), (\ref{ms-l2-err}) and  the above two inequalities give
\[
\begin{split}
&\: \| (u - u_{ms})(\cdot,T) \|^{2} +\int_0^T \| u - u_{ms} \|_{a}^{2} \\
\leq &\: CH^{2}\Lambda^{-1} \kappa_0^{-\frac{1}{2}} \|\kappa^{-\frac{1}{2}}(f-u_{t})\|_{L^2(0,T;L^2(\Omega))} \bigg(\|u_{0}\|_a^{2}+\|u_{0,ms}\|_a^{2} + \int_{0}^{T}\|f\|^{2} \bigg)^{\frac{1}{2}} \\
\quad &\: + CH^2\Lambda^{-1}\|\kappa^{-\frac{1}{2}}(f-u_{t})\|_{L^2(0,T;L^2(\Omega))}^2.
\end{split}
\]
This completes the proof.
\end{proof}

The following theorem gives the $L^2$ error estimate.
\begin{theorem}
\label{thm1}
 If $f_t \in L^1(0,T;L^2(\Omega))$ and $u_{tt} \in L^1(0,T;L^2(\Omega))$, then
\begin{equation}\label{th1}
\begin{split}
  \|(u-u_{ms})(T)\|&\leq CH^{2}\Lambda^{-1} \kappa_0^{-\frac{1}{2}} \bigg( \max_{0\leq t\leq T} \|\kappa^{-\frac{1}{2}}(f-u_{t})\|
  +\|\kappa^{-\frac{1}{2}}(f_{t}-u_{tt})\|_{L^1(0,T;L^2(\Omega))}\bigg) \\
  &\quad +\|u_{0}-u_{0,ms}\|,
   \end{split}
\end{equation}
where $C$ is a constant  independent of $\kappa$ and the mesh size $H$.
\end{theorem}
\begin{proof}
  We rewrite
 \[
 u-u_{ms}=(u-\widehat{u})+(\widehat{u}-u_{ms}):=\theta+\rho, \quad\forall t>0,
 \]
where $\widehat{u}$ is the elliptic projection in the space $V_{ms}$ of the exact solution $u$ and satisfies (\ref{ms-a-err}) and (\ref{ms-l2-err}).
By (\ref{ms-l2-err}), we have
\begin{equation}\label{equ1}
\begin{split}
\| \theta\| \leq CH^{2}\Lambda^{-1} \kappa_0^{-\frac{1}{2}} \|\kappa^{-\frac{1}{2}}(f-u_{t})\|, \quad\forall t>0.
\end{split}
\end{equation}
By equation (\ref{eq:scheme}) and (\ref{ms-ell}),  for all $v\in V_{ms}$ and any $t>0$, we have
\begin{equation}\label{inset1}
\begin{split}
((\widehat{u}-u_{ms})_{t},v)+a(\widehat{u}-u_{ms},v) 
&=(\widehat{u}_{t}, v)+a(u,v)-(f,v)\\
&=((\widehat{u}-u)_{t},v).
\end{split}
\end{equation}
Taking $v=\rho=\widehat{u}-u_{ms}\in V_{ms}$, it follows that
\[
(\rho_{t},\rho)+a(\rho,\rho)=((\widehat{u}-u)_{t},\rho).
\]
Since $a (\rho,\rho) \geq0$,   Cauchy-Schwarz inequality implies
\[
  \frac{1}{2}\frac{d}{dt}\|\rho\|^{2}=\|\rho\|\frac{d}{dt}\|\rho\| \leq \|(\widehat{u}-u)_{t}\| \, \|\rho\|.
\]
So,
\[
\frac{d}{dt}\|\rho\| \leq \|(\widehat{u}-u)_{t}\|.
\]
Integrating in time from $0$ to $T$, we have
\[
\| \rho(\cdot, T)\|   \leq  \| \rho(\cdot, 0)\| +   \int_0^T \|(\widehat{u}-u)_{t}\|.
\]
By (\ref{ms-l2-err}), we have
\begin{equation}\label{timebound}
\begin{split}
\|(\widehat{u}-u)_{t}\| \leq CH^{2}\Lambda^{-1} \kappa_0^{-\frac{1}{2}} \|\kappa^{-\frac{1}{2}}(f_t-u_{tt})\|, \quad\forall t>0.
\end{split}
\end{equation}
The above inequality  implies that
\begin{equation}\label{equ2}
\begin{split}
  \|(\widehat{u}-u_{ms})(T)\|&\leq  \|(\widehat{u}-u_{ms})(\cdot, 0)\|+  \int_{0}^{T}\|(\widehat{u}-u)_{t}\|  \\
  &\leq  \|u_{0}-u_{0,ms}\|+  \|\widehat u_{0}-u_{0}\|+CH^{2}\Lambda^{-1} \kappa_0^{-\frac{1}{2}} \int_{0}^{T}\|\kappa^{-\frac{1}{2}}(f_{t}-u_{tt})\|.
\end{split}
\end{equation}
Finally,  we have
\begin{equation}
\begin{split}
  \|(u-u_{ms})(\cdot, T)\|&\leq \|(u-\widehat{u})(\cdot, T)\|+\|(\widehat{u}-u_{ms})(\cdot,T)\|\\
  &\leq CH^{2}\Lambda^{-1} \kappa_0^{-\frac{1}{2}} \bigg( \max_{0\leq t\leq T} \|\kappa^{-\frac{1}{2}}(f-u_{t})\|
  +\int_{0}^{T}\|\kappa^{-\frac{1}{2}}(f_{t}-u_{tt})\|ds\bigg) \\
  &\quad +\|u_{0}-u_{0,ms}\|.
\end{split}
\end{equation}
where we have used (\ref{equ1}) and (\ref{equ2}). This completes the proof.
\end{proof}

Now, we ready to  consider the error estimates in the $L^{2}$ norm and energy norm using the relaxed  CEM-GMsFEM. The following lemma is necessary to derive the error estimates.

\begin{lemma}\label{lem21} \cite{Eric2017CEM}  Let $u$ be the solution of the elliptic equation (\ref{elliptic}) and $\omega\in V_{rms}$ be the elliptic projection of  $u$. If $\{\chi_{j}\}$ are the bilinear partition of unity functions and $m=O(\log(\frac{\kappa_1}{H\kappa_0}))$, where $m$ is the number of oversampling layers in the construction of basis in (\ref{uncon1}), then we have
\begin{equation}\label{elli21}
  \|u-\omega\|_{a}\leq CH\Lambda^{-\frac{1}{2}}\|\kappa^{-\frac{1}{2}}f\|.
\end{equation}
\end{lemma}
By using  Lemma \ref{lem21} and the similar arguments in the proof of Theorem \ref{thm2} and Theorem \ref{thm1}, we can derive the estimates for the relaxed CEM-GMsFEM for the parabolic equation.
\begin{theorem}
\label{thm21}
Let $u_{rms}\in V_{rms}$ be the approximate solution of the  solution $u$ defined in (\ref{eq-weak}) in the space $V_{rms}$. Then
\[
\begin{split}
&\: \| (u - u_{rms})(\cdot,T) \|^{2} +\int_0^T \| u - u_{rms} \|_{a}^{2} \\
\leq &\: CH^{2}\Lambda^{-1} \kappa_0^{-\frac{1}{2}} \|\kappa^{-\frac{1}{2}}(f-u_{t})\|_{L^2(0,T;L^2(\Omega))} \bigg(\|u_{0}\|_a^{2}+\|u_{0,rms}\|_a^{2} + \|f\|^{2}_{L^2(0,T;L^2(\Omega))} \bigg)^{\frac{1}{2}} \\
\quad &\: + CH^2\Lambda^{-1}\|\kappa^{-\frac{1}{2}}(f-u_{t})\|_{L^2(0,T;L^2(\Omega))}^2.
\end{split}
\]
where $C$ is a constant  independent of $\kappa$ and the mesh size $H$.
\end{theorem}

\begin{theorem}\label{thm22}
Let $u_{rms}\in V_{rms}$ be the approximate solution of the  solution $u$  defined in (\ref{eq-weak}).  Then
\[
\begin{split}
  \|u-u_{rms}\|&\leq CH^{2}\Lambda^{-1} \kappa_0^{-\frac{1}{2}} \bigg( \max_{0\leq t\leq T} \|\kappa^{-\frac{1}{2}}(f-u_{t})\|
  +\|\kappa^{-\frac{1}{2}}(f_{t}-u_{tt})\|_{ L^1(0,T;L^2(\Omega))    }\bigg) \\
  &\quad +\|u_{0}-u_{0,rms}\|.
   \end{split}
\]
where $C$ is a constant  independent of $\kappa$ and $H$.
\end{theorem}

\subsection{A posteriori error bound}
In addition to a priori error estimates,  a posteriori error estimation is also very important. It can give a computable error bound to access the quality of the numerical solution.
It also serves as a foundation for an adaptive basis enrichment strategy.
Now we derive a posteriori error bound of CEM-GMsFEM for the parabolic equation.
\begin{theorem}\label{thmpost}
Let $u$ be the solution of the parabolic equation (\ref{eq-weak}) and $u_{ms}$  the approximate solution of  $u$ in the space $V_{ms}$. Then
\begin{equation}\label{posterior1}
\|(u-u_{ms})(T)\|^{2}+\int_{0}^{T}\|u-u_{ms}\|_{a}^{2}\leq 2M(1+\Lambda^{-1})\sum_{n}\sum_{i=1}^{N_c}\|R_{i}^{n}\|_{Q^{*}}^{2}+\|(u-u_{ms})(0)\|^{2},
\end{equation}
where the local residual $R_{i}^{n}$ is defined as
\[
R_{i}^{n}(v)=\int_{t_{n}}^{t_{n+1}}\Big\{(f,v)-((u_{ms})_{t},v)-a(u_{ms},v)\Big\}, \quad v\in L^2(t_n,t_{n+1};V_{0}(\omega_{i})),
\]
and the residual norm is defined by
\[
\|R_{i}^{n}\|_{Q^*}:=\sup_{v\in L^{2}(t_n,t_{n+1};V_{0}(\omega_{i}))}\frac{R_{i}^{n}(v)}{\|v\|_{L^{2}(t_n,t_{n+1};V_{0}(\omega_{i}))}}.
\]
The constant $M$ is the maximum number of overlapping coarse neighborhoods.
\end{theorem}
\begin{proof}
First of all, we have
\begin{equation}\label{ac}
  \begin{split}
  \frac{1}{2}\frac{d}{dt}\|u-u_{ms}\|^{2}+\|u-u_{ms}\|_{a}^{2}&=((u-u_{ms})_{t},u-u_{ms})+a(u-u_{ms},u-u_{ms})\\
  &=(f,u-u_{ms})-((u_{ms})_{t},u-u_{ms})-a(u_{ms},u-u_{ms}) \\
  &=(f,z)-((u_{ms})_{t},z)-a(u_{ms},z),
  \end{split}
\end{equation}
where $z=u-u_{ms}$. Let $\widehat{z}$ be the spectral projection of $z$. Using (\ref{eq:scheme}), we have
\begin{equation}
 \frac{1}{2}\frac{d}{dt}\|u-u_{ms}\|^{2}+\|u-u_{ms}\|_{a}^{2} = (f,z-\widehat{z})-((u_{ms})_{t},z-\widehat{z})-a(u_{ms},z-\widehat{z}).
\end{equation}
Integrating with respect to time from $t_n$ to $t_{n+1}$, we have
\begin{eqnarray*}
  &&\frac{1}{2}\|(u-u_{ms})(\cdot,t_{n+1})\|^{2}+\int_{t_n}^{t_{n+1}}\|u-u_{ms}\|_{a}^{2}\\
  &&= R^n(z-\widehat{z}) +\frac{1}{2}\|(u-u_{ms})(\cdot,t_n)\|^{2},
\end{eqnarray*}
where we define the global residual $R^n$  by
\[
R^n(v)=\int_{t_n}^{t_{n+1}}\{(f,v)-((u_{ms})_{t},v)-a(u_{ms},v)\},
\]
which is a linear functional on $Q^n:=L^{2}((t_n,t_{n+1});V_{0})$.
We can rewrite
\[
R^n(v)=R^n(\sum_{i=1}^{N_c} \chi_{i}v)=\sum_{i=1}^{N_c}\int_{t_n}^{t_{n+1}}\{(f,\chi_{i}v)-((u_{ms})_{t},\chi_{i}v)-a(u_{ms},\chi_{i}v)\}.
\]
Using the definition of $R^n_i$, we have $R^n(v)=\sum_{i=1}^{N_c} R^n_{i}(\chi_{i}v)$.

Notice that
\[
R^n_{i}(\chi_{i}(z-\widehat{z}))\leq\|R^n_{i}\|_{Q^{*}}\|\chi_{i}(z-\widehat{z})\|_{Q},
\]
where $\|\cdot\|_Q$ denotes the norm in $Q^n$.
In addition, we have
\begin{equation*}
\|\chi_{i}(z-\widehat{z})\|_{Q}^2
= \int_{t_n}^{t_{n+1}} \| \chi_i(z-\widehat{z})\|_{a_i}^2
\leq 2\int_{t_n}^{t_{n+1}} ( \| z-\widehat{z}\|_{a_i}^2 + \|z-\widehat{z}\|_{s_i}^2)
\leq 2(1+\Lambda^{-1}) \int_{t_n}^{t_{n+1}}  \| z\|_{a_i}^2,
\end{equation*}
where $\|\cdot\|_{a_i}$ and $\|\cdot\|_{s_i}$ denote the $a$-norm and $s$-norm restricted in $\omega_i$.
Therefore,
\begin{eqnarray*}
  &&\frac{1}{2}\|(u-u_{ms})(\cdot,t_{n+1})\|^{2}+\int_{t_n}^{t_{n+1}}\|u-u_{ms}\|_{a}^{2}\\
  &&\leq (2(1+\Lambda^{-1}))^{\frac{1}{2}}\sum_{i=1}^{N_c}  \|R^n_{i}\|_{Q^{*}} \|u-u_{ms}\|_{Q_i} +\frac{1}{2}\|(u-u_{ms})(\cdot,t_n)\|^{2},
\end{eqnarray*}
where $\|u-u_{ms}\|_{Q_i}^2 = \int_{t_n}^{t_{n+1}} \|u-u_{ms}\|_{a_i}^{2}$.
Using Cauchy Schwarz inequality, we have
\begin{eqnarray*}
  &&\frac{1}{2}\|(u-u_{ms})(\cdot,t_{n+1})\|^{2}+ \frac{1}{2} \int_{t_n}^{t_{n+1}}\|u-u_{ms}\|_{a}^{2}\\
  &&\leq M(1+\Lambda^{-1}) \sum_{i=1}^{N_c}  \|R^n_{i}\|_{Q^{*}}^2 +\frac{1}{2}\|(u-u_{ms})(\cdot,t_n)\|^{2}.
\end{eqnarray*}
The proof is completed by summing the above over all $n$.
\end{proof}
\section{Full discrete schemes}

In addition to spatial variable discretization, we also need to discretize the temporal variable to solve the parabolic equation.
In the section, we present two widely  used schemes: backward Euler and forward Euler.
For the presentation, we use the following notations, for any $n=0,1,2,\cdots$,
\[
 t^{n}:=n\triangle t, \quad  u^{n}:=u(t^{n},x).
\]

\subsection{Backward Euler}
We first consider the backward Euler method. Let $\triangle t$ is the time step and $V_{ms}$ be the multiscale finite element space. And
let $U_{ms}^{n}:=u_{ms}(t_{n})$. Then the fully discrete formulation using backward Euler method is
\begin{equation}\label{full}
\left\{
 \begin{aligned}
  (\frac{U_{ms}^{n}-U_{ms}^{n-1}}{\Delta t},v)+a(U_{ms}^{n},v)&=(f(t_{n}),v)\quad\quad\quad \forall v\in V_{ms}\\
   U_{ms}^{0}&=u_{0,ms}.
\end{aligned}
   \right.
\end{equation}

We can  show that the backward Euler method is unconditionally stable.
 In the equations (\ref{full}), we  choose  $v=U_{ms}^{n}$. Because  $a(U_{ms}^{n},U_{ms}^{n})\geq 0$,
\begin{equation}\label{sta}
  \|U_{ms}^{n}\|^{2}-(U_{ms}^{n-1},U_{ms}^{n})\leq\Delta t\|f^{n}\|\cdot\|U_{ms}^{n}\|,
\end{equation}
where $f^n = f(t_n)$.
By  Cauchy-Schwarz inequality,
\[
 (U_{ms}^{n-1},U_{ms}^{n})\leq\|U_{ms}^{n-1}\|\|U_{ms}^{n}\|,
\]
 the inequality (\ref{sta}) implies
\[
  \|U_{ms}^{n}\|\leq\|U_{ms}^{n-1}\|+\Delta t\|f^{n}\|.
\]
By repeating  the above inequality, we have
\begin{equation}\label{stable}
  \|U_{ms}^{n}\|\leq\|U_{ms}^{0}\|+\Delta t\sum_{j=1}^{n}\|f^{j}\|.
\end{equation}
This shows that the backward Euler method is unconditionally stable.

The following theorem gives the convergence result for the fully discrete scheme using backward Euler.
\begin{theorem}\label{thm5}
Let $U_{ms}^{n}$ and $u$ be the solution of (\ref{full}) and (\ref{ex-eq}), respectively.  Then
\[
 \|U_{ms}^{n}-u(t_{n})\|\leq  CH^{2}\Lambda^{-1}\kappa_{0}^{-\frac{1}{2}} \|\kappa^{-\frac{1}{2}}(f^{n}-u_{t}(t_{n})\|+\Delta t\int_{0}^{t_{n}}\|u_{tt}\|ds+\|u_{0,ms}-u_{0}\|.
\]
\end{theorem}
\begin{proof}
Let $\widehat{u}$ be the elliptic projection of $u$. Then
\[
  U_{ms}^{n}-u(t_{n})=(U_{ms}^{n}-\widehat{u}(t_{n}))+(\widehat{u}(t_{n})-u(t_{n}))=\theta^{n}+\rho^{n}
\]
From the inequality (\ref{ms-l2-err}), we have
\[
  \|\rho^{n}\|=\|\widehat{u}(t_{n})-u(t_{n})\|  \leq C H^{2}\Lambda^{-1}\kappa_{0}^{-\frac{1}{2}} \|\kappa^{-\frac{1}{2}}(f^{n}-u_{t}(t_{n})\|.
\]
 By equation (\ref{ex-eq}) and (\ref{ms-ell}),  for any $v\in V_{ms}$
\begin{equation}\label{inset2}
\begin{split}
((\widehat{u}-u_{ms})_{t},v)+a(\widehat{u}-u_{ms},v)
&=(\widehat{u}_{t},v)+a(u,v)-(f,v)\\
&=((\widehat{u}-u)_{t},v).
\end{split}
\end{equation}
Because  $\partial_{t} U_{ms}^{n}=\frac{U_{ms}^{n}-U_{ms}^{n-1}}{\Delta t}$,  we can rewrite  (\ref{inset2}) as
\begin{equation}\label{infull}
  (\partial_{t} \theta^{n},v)+a(\theta^{n},v)=(u_{t}(t_{n})-\partial_{t}\widehat{u}^{n},v).
\end{equation}
Due to
\begin{equation}\label{fload}
\begin{split}
u_{t}(t_{n})-\partial_{t} \widehat{u}^{n} &=(u_{t}(t_{n})-\partial_{t}u^{n})+(\partial_{t}u^{n}-\partial_{t} \widehat{u}^{n})\\
  &:=w_{1}^{n}+w_{2}^{n},
\end{split}
\end{equation}
combining  (\ref{stable}), (\ref{infull}) and (\ref{fload}) gives
\begin{equation}
  \|\theta^{n}\|\leq\|\theta^{0}\|+\Delta t\sum_{j=1}^{n}\|w_{1}^{j}\|+\Delta t\sum_{j=1}^{n}\|w_{2}^{j}\|.
\end{equation}
If we choose $U_{ms}^{0}=u_{0,ms}$ and apply  (\ref{ms-l2-err}), we have
\[
  \|U_{ms}^{0}-\widehat{u}_{0}\|\leq\|u_{0,ms}-u_{0}\|+\|u_{0}-\widehat{u}_{0}\|.
\]
By Taylor's formula, it follows that
\[
  w_{1}^{j}=u_{t}(t_{j})-(\Delta t)^{-1}\big(u(t_{j})-u(t_{j-1})\big)=(\Delta t)^{-1}\int_{t_{j-1}}^{t_{j}}(s-t_{j-1})u_{tt}(s)ds.
\]
Then
\[
  \Delta t\sum_{j=1}^{n}\|w_{1}^{j}\|\leq\sum_{j=1}^{n}\|\int_{t_{j-1}}^{t_{j}}(s-t_{j-1})u_{tt}(s)ds\|\leq\Delta t\int_{0}^{t_{n}}\|u_{tt}\|ds.
\]
Because
\[
  w_{2}^{j}=\partial_{t}u^{j} -\partial_{t}\widehat{u}^{j}=(\Delta t)^{-1}\int_{t_{j-1}}^{t_{j}}(u_{t}-\widehat{u}_{t})ds,
\]
we have
\begin{equation}
\begin{split}
  \Delta t \sum_{j=1}^{n}\|w_{2}^{j}\|&\leq\sum_{j=1}^{n}(\|u(t_{j})-\widehat{u}(t_{j})\|-\|u(t_{j-1})-\widehat{u}(t_{j-1})\|)\\
  &\leq\|u(t_{n})-\widehat{u}(t_{n})\|-\|u(0)-\widehat{u}(0)\|.
\end{split}
\end{equation}
Hence,
\[
  \|U_{ms}^{n}-u(t_{n})\|\leq  CH^{2}\Lambda^{-1}\kappa_{0}^{-\frac{1}{2}} \|\kappa^{-\frac{1}{2}}(f^{n}-u_{t}(t_{n})\|+\Delta t\int_{0}^{t_{n}}\|u_{tt}\|ds+\|u_{0,ms}-u_{0}\|.
\]
\end{proof}

\subsection{Forward Euler discretization}
Using the forward Euler method, we have the following weak formulation for (\ref{ex-eq})
\begin{equation}\label{forward}
\left\{
 \begin{aligned}
  (\frac{U_{ms}^{n+1}-U_{ms}^{n}}{\Delta t},v)+a(U_{ms}^{n},v)&=(f(t_{n}),v)\quad\quad \forall v\in V_{ms},n\geq 1\\
   U_{ms}^{0}&=u_{0,ms}.
\end{aligned}
   \right.
\end{equation}
If we define
\[
V_{ms}=\text{span}\big\{\varphi_{k,ms}|1\leq k \leq \sum_{i=1}^{N} L_{i}\big\}=\text{span}\{\psi_{j,ms}^{(i)}|1\leq j\leq L_{i},1\leq i\leq N\},
\]
then  $U_{ms}^{n}$ and $f(t_{n})$ can be expressed as follows,
 \[
 U_{ms}^{n}=\sum_{k}\alpha_{k}^{n}\varphi_{k,ms},\quad f(t_{n})=\sum_{k} b_{k}^{n}\varphi_{k,ms}.
 \]
We can express the equation (\ref{forward}) in the following  matrix form,
\begin{equation}\label{matrix}
  M\alpha^{n+1}=(M-\Delta t A)\alpha^{n}+\Delta t b^{n},
\end{equation}
where
\begin{eqnarray*}
&&(\alpha^{n})_{i}=\alpha_{i}^{n},\quad (b^{n})_{i}=b_{i}^{n},\\
&&(M)_{i,j}=\int_{\Omega}\varphi_{i,ms}\varphi_{j,ms},\quad
(A)_{i,j}=\int_{\Omega}\kappa\nabla\varphi_{i,ms}\cdot\nabla\varphi_{j,ms}.
\end{eqnarray*}
Using the equation (\ref{matrix}) and recursion, we obtain
\begin{equation}\label{recu}
  \alpha^{n}=(I-\Delta t M^{-1}A)^{n}\alpha^{0}+\Sigma_{i=0}^{n}(I-\Delta t M^{-1}A)^{i}\Delta t M^{-1}b^{n-i}¡£
\end{equation}
It is easy to prove that  the equation (\ref{recu}) is stable if and only if $|1-\Delta t \lambda|\leq1$, where $\lambda$ is the eigenvalue of
\[
a(\varphi_{i,ms},v)=\lambda(\varphi_{i,ms},v), \quad \forall v\in V_{ms}.
\]
We remark that the analysis for the convergence is similar to that for the backward Euler scheme, so we omit it.

\section{Numerical results}
In this section, we present some numerical results by using  CEM-GMsFEM to solve the parabolic equation with the multiscale permeability field  $\kappa(x)$.
For time discretization, we will use the backward Euler scheme. In Section \ref{s1}, we use  CEM-GMsFEM to solve the parabolic equation with different multiscale permeability fields. In Section \ref{s2}, the numerical results using the CEM-GMsFEM will be given for the parabolic equation in fracture porous media. In Section \ref{s3}, we verify the posteriori error bound by numerical examples.

For the numerical examples, we consider the parabolic equation (\ref{ex-eq}) with the homogeneous Dirichlet boundary condition, source term $f(x,t)$ and initial condition $u_{0}(x)$ are defined by
\[
\begin{aligned}
f(x,t)&=3\pi^{2}\exp(\pi^{2}t)\sin(\pi x_{1})\sin(\pi x_{2}), \quad x\in [0,1]^{2},\\
u_{0}(x)&=\sin(\pi x_{1})\sin(\pi x_{2}), \quad x\in [0,1]^{2}.
\end{aligned}
\]
Here $x:=(x_{1},x_{2})$. In the spatial direction, the computational domain is $\Omega=[0,1]^2$ and we use $200\times 200$ fine grid to compute a reference solution. We will choose different coarse grid size to compute the relative error between the fine-scale solution and CEM-GMsFEM solution. Two high contrast permeability fields $\kappa(x)$ used in the numerical examples are shown in Figure \ref{cof}.
\begin{figure}[htbp]
\centering
\subfigure[$ \kappa_{1}$]{\includegraphics[width=2.5in, height=2in]{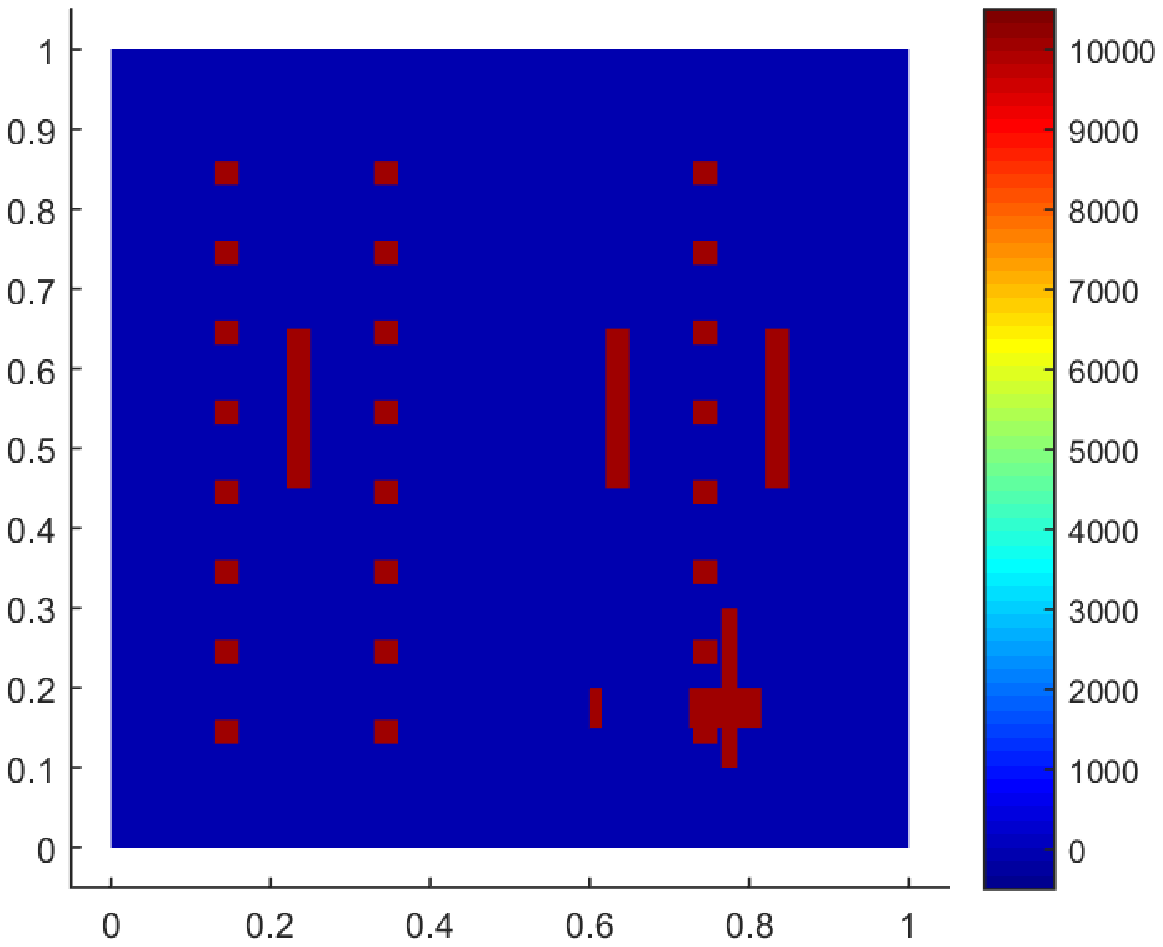}}
\subfigure[$ \kappa_{2}$]{\includegraphics[width=2.5in, height=2in]{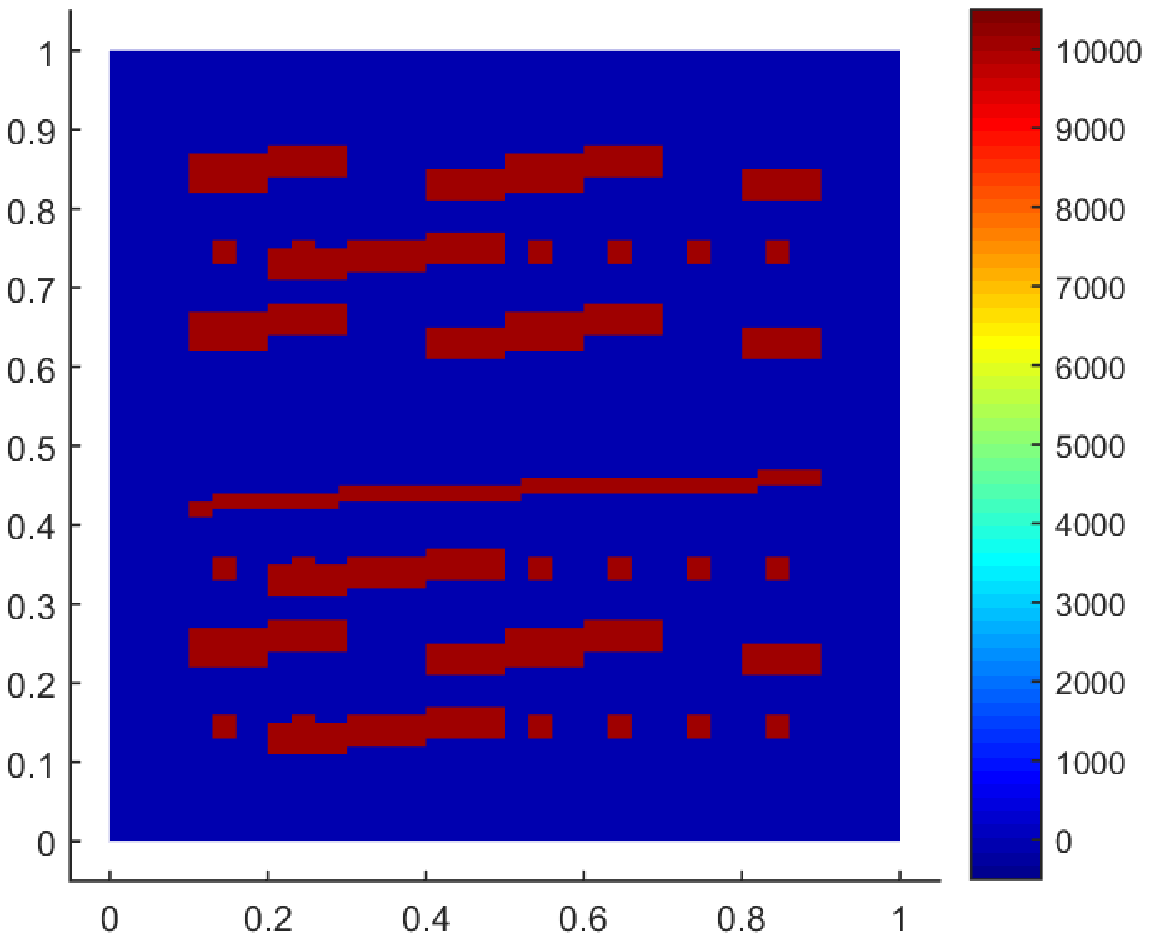}}
\caption{ Two permeability fields.}
\label{cof}
\end{figure}

To measure the approximation accuracy,  the relative errors between reference solution and multiscale method are utilized and defined as follows,
\[
\varepsilon=\frac{\|u_{h}(x,T)-u_{ms}(x,T)\|_{L^{2}(\Omega)}}{\|u_{h}(x,T)\|_{L^{2}(\Omega)}},
\]
\[
\varepsilon_{a}=\frac{\|u_{h}(x,T)-u_{ms}(x,T)\|_{a(\Omega)}}{\|u_{h}(x,T)\|_{a(\Omega)}},
\]
where $u_{h}$ is the fine-scale FEM solution in fine grid and $u_{ms}$ is  CEM-GMsFEM solution.

\subsection{Numerical results for CEM-GMsFEM}\label{s1}
In this subsection, we present the numerical results using CEM-GMsFEM to solve the equation (\ref{ex-eq}) with different permeability fields. In  Example 1, the permeability field $\kappa_1$ depicted in Figure \ref{cof} (left) will be used. The contrast ratio   is $10^4$ in the permeability field $\kappa_1$. In the $V_{ms}$ space, the number of oversampling layers used is  $4*(\log(H)/\log(1/10))$.
In Figure \ref{errtu1}, we present the results for the first example using different numbers of local basis functions when fixing  the coarse grid size $H=1/10$. From the figure, we obtain two observations: (1) the error decreases as the number of local basis functions increases; (2) when the number of local basis functions exceeds a certain number, the error decreases very slowly as the eigenvalues of the local spectral problems  decay slowly. So we can get a good accuracy by using only a few local  basis functions on each coarse block. In this example, four local basis functions for each coarse block can achieve very good accuracy. Figure \ref{wutu1} shows the relative error $\varepsilon_{a}$ with different numbers of oversampling layers when we fix the coarse grid size $H=1/10$ and use four basis functions on each coarse block. By the figure, we can see  that as the number of oversampling layer increases, the approximation becomes more accurate and the error decreases very slowly once the number of oversampling layers attains  a certain number. To show the relationship between the error and the coarse grid size $H$,  Table \ref{exm-tab1} lists  the numerical results $(T=1,\Delta t=0.01)$ with different coarse grid sizes $H$. It shows that: (1) as the the coarse grid becomes finer, the error decreases; (2) the CEM-GMsFEM for the parabolic equation has a convergence rate  proportional to the coarse grid size $H$. Based on the above theoretical analysis, the error depends on $H$ and $\Lambda$. Different sizes of the coarse grid may result in  different $\Lambda$. Given a fixed  $\Lambda$,  the $L^{2}$ norm  error can achieve  a second order convergence rate and the energy norm error has a first order convergence rate. Figure \ref{exm-solu1} shows the fine-scale FEM solution and the CEM-GMsFEM solution at times $T=1$, which show almost the same solution profile.

\begin{figure}
  \centering
  \includegraphics[width=3in]{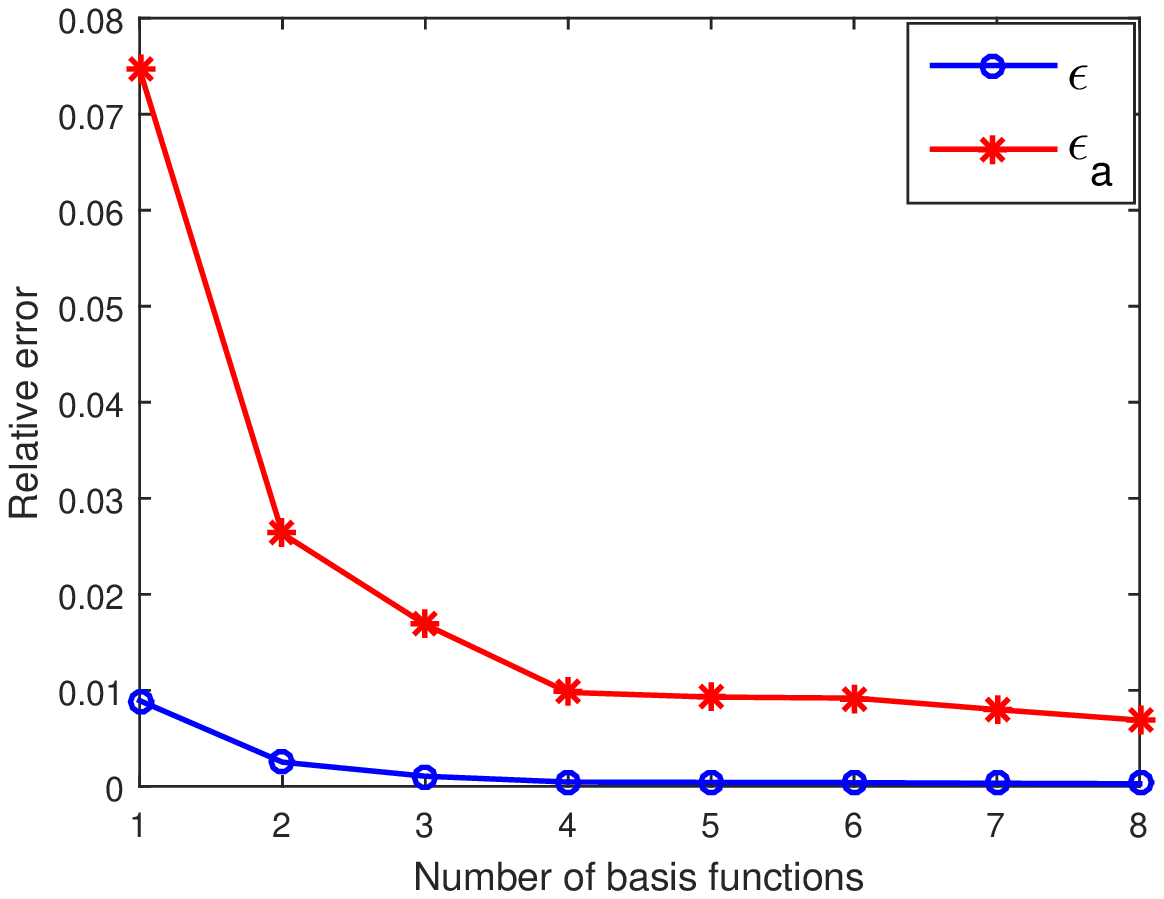}
  \caption{Numerical result with different  numbers of local basis functions (Example 1)}\label{errtu1}
\end{figure}

\begin{figure}
  \centering
  \includegraphics[width=3in]{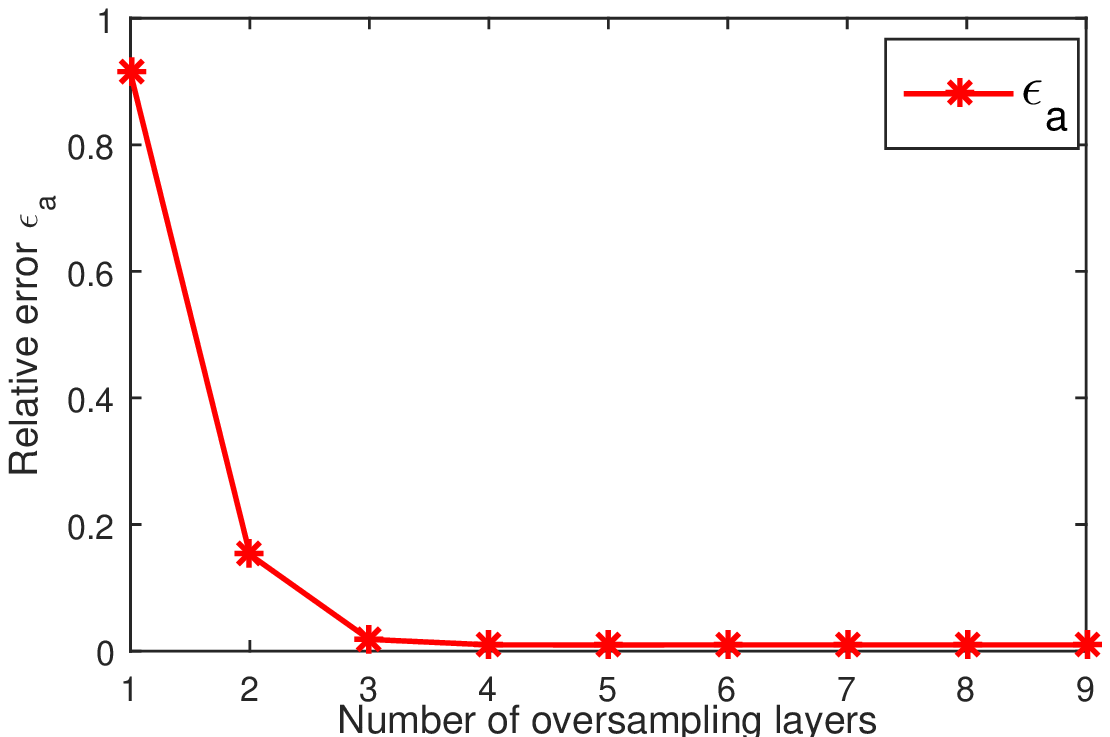}
  \caption{Numerical result with different numbers of oversampling layers (Example 1)}\label{wutu1}
\end{figure}

\begin{table}[hbtp]
\centering
\caption{Numerical result with varying coarse grid size $H$ and oversampling coarse layers $m$ for  Example 1 $(\kappa=\kappa_{1})$}
\vspace{2pt}
\begin{tabular}{|c|c|c|c|c|}
  \hline
   Basis \# per element & $H$ & $m$  & $\varepsilon$& $\varepsilon_{a}$ \\
  \hline
  $4$ & $1/10$ & $4$ & $4.5137E-04$& $9.7892E-03$  \\
  \hline
  $4$ & $1/20$ &$6$  & $1.1483E-04$ & $3.7640E-03$\\
  \hline
  $4$ & $1/40$ & $7$ & $1.7609E-05$& $1.5141E-03$ \\
  \hline
\end{tabular}
\label{exm-tab1}
\end{table}

\begin{figure}
  \centering
  \includegraphics[width=4in]{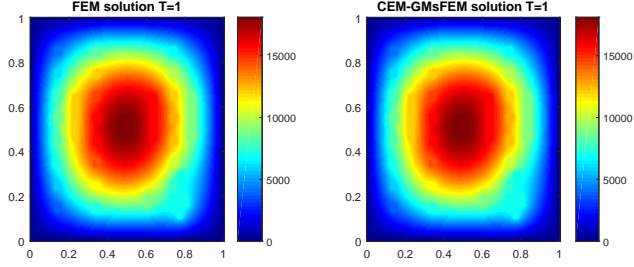}
  \caption{The left is the reference solution, the right is the CEM-GMsFEM solution.  $H=\frac{1}{10}$,  $4$  oversampling layers and  $4$ local multiscale basis functions,  $T=1$}\label{exm-solu1}
\end{figure}

In the second example, we use the same equation but with a different permeability field $\kappa_{2}$ shown in Figure \ref{cof} (right) and the contrast is still $10^{4}$.  The fine grid size, the coarse grid size and the oversampling layer are the same as the previous example. In Figure \ref{errtu2}, we present the results for using different number of local basis functions when we fix the coarse grid size $H=1/10$. When increasing  the number of basis functions on each  coarse grid, the error will decrease. As the same as  Example 1,  four basis functions for each coarse block can give a very good approximation. In Figure \ref{wutu2}, the relative error $\varepsilon_{a}$ with different numbers of oversampling layers is given. Table \ref{exm-tab2} shows the numerical results with different coarse grid sizes $H$.

\begin{figure}
  \centering
  \includegraphics[width=3in]{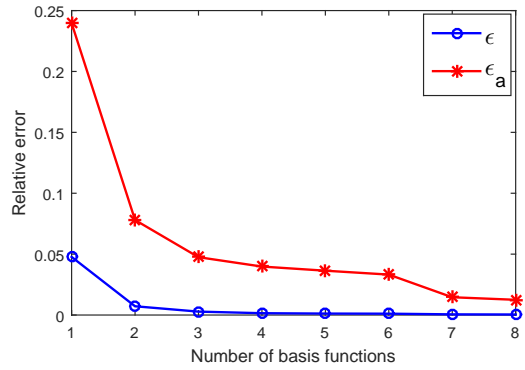}
  \caption{Numerical result with different  numbers of basis functions (Example 2)}\label{errtu2}
\end{figure}

\begin{figure}
  \centering
  \includegraphics[width=3in]{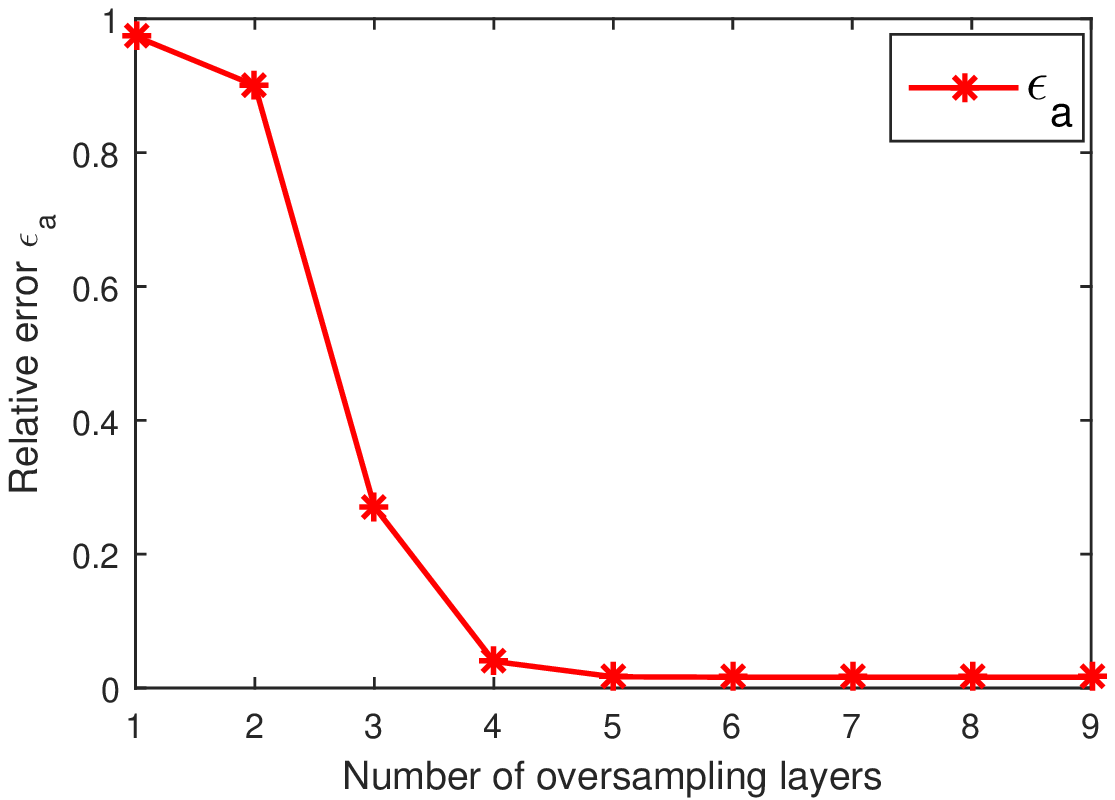}
  \caption{Numerical result with different  numbers of oversampling layers (Example 2)}\label{wutu2}
\end{figure}

\begin{table}[hbtp]
\centering
\caption{Numerical result with different  coarse grid sizes $H$  and oversampling coarse layers $m$ for  Example 2 $(\kappa=\kappa_{2})$}
\vspace{2pt}
\begin{tabular}{|c|c|c|c|c|}
  \hline
   Basis \# per element & $H$ & $m$  & $\varepsilon$ & $\varepsilon_{a}$\\
  \hline
  $4$ & $1/10$ & $4$ & $1.5300E-03$ & $3.9793E-02$\\
  \hline
  $4$ & $1/20$ &$6$  & $2.5246E-04$& $7.2485E-03$  \\
  \hline
  $4$ & $1/40$ & $7$ & $3.3741E-05$ & $2.7634E-03$\\
  \hline
\end{tabular}
\label{exm-tab2}
\end{table}

\subsection{Application to a fractured medium}\label{s2}
In this subsection, we consider a single phase flow in a fracture porous medium.  The permeability value in fracture is  much  larger  than that in the surrounding medium. We use the discrete fracture model (DFM) introduce in \cite{frac1} and apply CEM-GMsFE to the fracture model. The main idea is to treat the fracture as interface between subdomains. Thus, the spatial domain is divided into  matrix and  fractures, where the matrix is  two dimensional and the fractures are one dimensional, i.e.,
\[
\Omega=\Omega_{m}\cup(\cup_{i}\Omega_{\f,i}),
\]
where the subscript $m$ represents the matrix regions and $\f$ represents the fracture regions. So we can rewrite the weak formulation  corresponding to (\ref{ex-eq}) as
\begin{eqnarray}\label{fracture}
\begin{split}
  &\int_{\Omega}u_{t}\cdot v+\int_{\Omega}\kappa \nabla u \cdot \nabla v \\
  &=\int_{\Omega_{m}}u_{t}\cdot v+\sum_{i}\int_{\Omega_{\f,i}}u_{t}\cdot v
  +\int_{\Omega_{m}}\kappa \nabla u \cdot \nabla v+\sum_{i}\int_{\Omega_{\f,i}}\kappa \nabla u \cdot \nabla v\\
  &=\int_{\Omega}fv.
\end{split}
\end{eqnarray}
Using  CEM-GMsFEM, we take into account the fracture distributions in constructing the auxiliary space and solving the local minimizing problems.
 We show the computation domain and fracture  in  Figure \ref{frac1}. In the fracture domain, the permeability is $10^4$. In the matrix domain, the permeability is $1$. In the discrete fracture model, we use the constraint energy minimizing in the $V_{ms}$ space. We show the numerical solution in $(T=1,\Delta t=0.01 )$ in Figure \ref{frac-solu}. In this figure, the domain has three fractures and the reference solution is computed by the FEM on $160\times160$ fine grid. We use four local multiscale  basis functions on each  coarse block  and the number of  oversampling layers is 4 to compute the solution of CEM-GMsFEM.
 The figure demonstrates that the CEM-GMsFEM gives  a good  approximation  solution for the case of the  fracture porous medium.
Table \ref{frac-reu} gives the numerical results $(T=1,\Delta t=0.01)$ with different coarse grid sizes $H$. By this table, we observe that the error significantly  decreases as the size of the coarse grid gets finer.

\begin{figure}
  \centering
  \includegraphics[width=2in,height=2in]{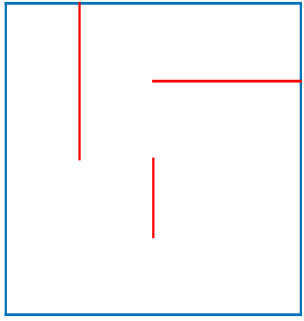}
  \caption{Two-dimensional matrix $\Omega_{m}$ with one-dimensional fracture $\gamma$}\label{frac1}
\end{figure}

\begin{figure}
  \centering
  \includegraphics[width=4in]{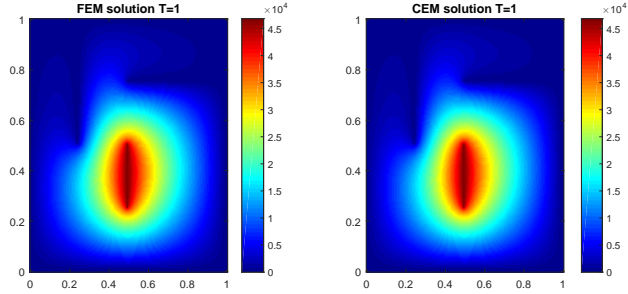}
  \caption{The left is the reference  solution, the right is the CEM-GMSFEM solution.  $H=\frac{1}{8}$, the oversampling layers is 4 and using  4 local multiscale  basis function. $(T=1)$}\label{frac-solu}
\end{figure}

\begin{table}[hbtp]
\centering
\caption{Numerical result with varying coarse grid size $H$ and oversampling coarse layers $m$}
\vspace{2pt}
\begin{tabular}{|c|c|c|c|c|}
  \hline
   Basis \# per element & $H$ & $m$  & $\varepsilon$& $\varepsilon_{a}$  \\
  \hline
  $4$ & $1/4$ & $3$ & $ 2.5751E-02$& $1.4806E-02$  \\
  \hline
  $4$ & $1/8$ &$4$  & $7.1869E-03$ & $5.6841E-03$\\
  \hline
  $4$ & $1/16$ & $5$ & $2.0709E-03$& $2.1280E-03$ \\
  \hline
\end{tabular}
\label{frac-reu}
\end{table}


\subsection{A posteriori error bound}\label{s3}
In this section, we present the numerical results to confirm the posteriori error bound in Theorem \ref{thmpost}. For simplicity of presentation, we define the  notations as follows
\[
\varepsilon_{L}:=\|(u-u_{ms})(T)\|^{2}+\int_{0}^{T}\|u-u_{ms}\|_{a}^{2},\quad \varepsilon_{R}:=\sum_{n}\sum_{i}\|R_{i}^{n}\|_{Q^{*}}^{2}+\|(u-u_{ms})(0)\|^{2},
\]
where $u$ is the reference solution and $u_{ms}$ is the CEM-GMsFEM solution.
In order to compute $\|R_{i}^{n}\|_{Q^{*}}$, we solve the problem: find $\phi_{i}^{n}\in V_{0}(\omega_{i})$ such that
\[
  a(\phi_{i}^{n},v)=(f,v)-(\frac{u_{ms}^{n+1}-u_{ms}^{n}}{\Delta t},v)-a(u_{ms}^{n+1},v),\quad \forall v\in V_{0}(\omega_{i}).
\]
Thus,
\[
  \| R_{i}^{n}\|_{Q^{*}}=(\Delta t)^{\frac{1}{2}}\|\phi_{i}^{n}\|_{a}.
\]

We consider the parabolic equation (\ref{ex-eq}) with the homogeneous Dirichlet boundary condition and zero  initial condition, the source term
\[
f(x,t)=t^{2}+(x_{1}+x_{2})^{2},\quad x\in[0,1]^{2}.
\]
In the example, we take the permeability field $\kappa_{2}$  depicted in Figure \ref{cof} (right). We use $200\times 200$ fine grid in the spatial domain  and the time step $\Delta t =0.01$. Due to the eigenvalue problem has three small eigenvalues in each coarse cell, we will use three multiscale  basis functions in each coarse cell. Table \ref{posterr} shows the numerical result with different coarse grid sizes. By the table, we observe that $\sum_{n}\sum_{i}\|R_{i}^{n}\|_{Q^{*}}^{2}+\|(u-u_{ms})(0)\|^{2}$ is a suitable a posteriori error estimate and the ratio $\frac{\varepsilon_R}{\varepsilon_L}$ keeps stable as coarse grid size varies.
To see how the local permeability structure effects on the local residual of CEM-GMsFEM, we arrange the local residuals $\{\Sigma_{n}\|R_{i}^{n}\|_{Q^*}^{2}\}_{i=1}^{N_{c}}$ in  ascending order, where $N_{c}$ is the number of coarse vertices. When  $N_{c}=121$, we plot  the permeability profile  in some regions $\omega_{i}$ and compute   $\Sigma_{n}\|R_{i}^{n}\|_{Q^*}^{2}$ in these local regions in Figure \ref{cofNx4}.
For $N_{c}=441$, we plot the local permeability profiles in some region $\omega_{i}$ in Figure \ref{cofNx8} and compute the value of $\Sigma_{n}\|R_{i}^{n}\|_{Q^*}^{2}$ in these regions.
In Figure \ref{cofNx4}, we observe that  the local permeability map $\kappa(\omega_{25})$  in the middle plot  is more complex than the right map $\kappa(\omega_{39})$, but the local residual $\Sigma_{n}\|R_{25}^{n}\|_{Q^*}^{2}$ is much smaller than $\Sigma_{n}\|R_{39}^{n}\|_{Q^*}^{2}$.
 In Figure \ref{cofNx8}, the  middle plot $\kappa(\omega_{71})$ is much more simple  than the left plot $\kappa(\omega_{68})$, but the local residual $\Sigma_{n}\|R_{71}^{n}\|_{Q^*}^{2}$ is larger than  $\Sigma_{n}\|R_{68}^{n}\|_{Q^*}^{2}$.  These results imply that the value of the local residual may not depend on the local pattern of the medium, but depends on the nonlocal structure of the medium.

\begin{table}[hbtp]
\centering
\caption{A posteriori error bound}
\vspace{2pt}
\begin{tabular}{|c|c|c|c|c|}
  \hline
    $H$ & oversampling layers  & $\varepsilon_{L}$ & $\varepsilon_{R}$  &   $\frac{\varepsilon_R}{\varepsilon_L}$   \\
  \hline
  $1/10$ & $4$& $8.4905E-05$ & $3.1505E-04$&$3.7107$\\
  \hline
 $1/20$ &$6$  & $3.5436E-06$& $1.1161E-05$ &$3.1496$ \\
  \hline
  $1/40$ & $7$ & $ 3.8295E-07 $ & $1.1989E-06$ &$ 3.1306$\\
  \hline
\end{tabular}
\label{posterr}
\end{table}

\begin{figure}[htbp]
\centering
\subfigure[\tiny $\Sigma_{n}\|R_{14}^{n}\|_{Q^*}^{2}=7.3808E-08$.]{\includegraphics[width=1.6in, height=1.7in,angle=0]{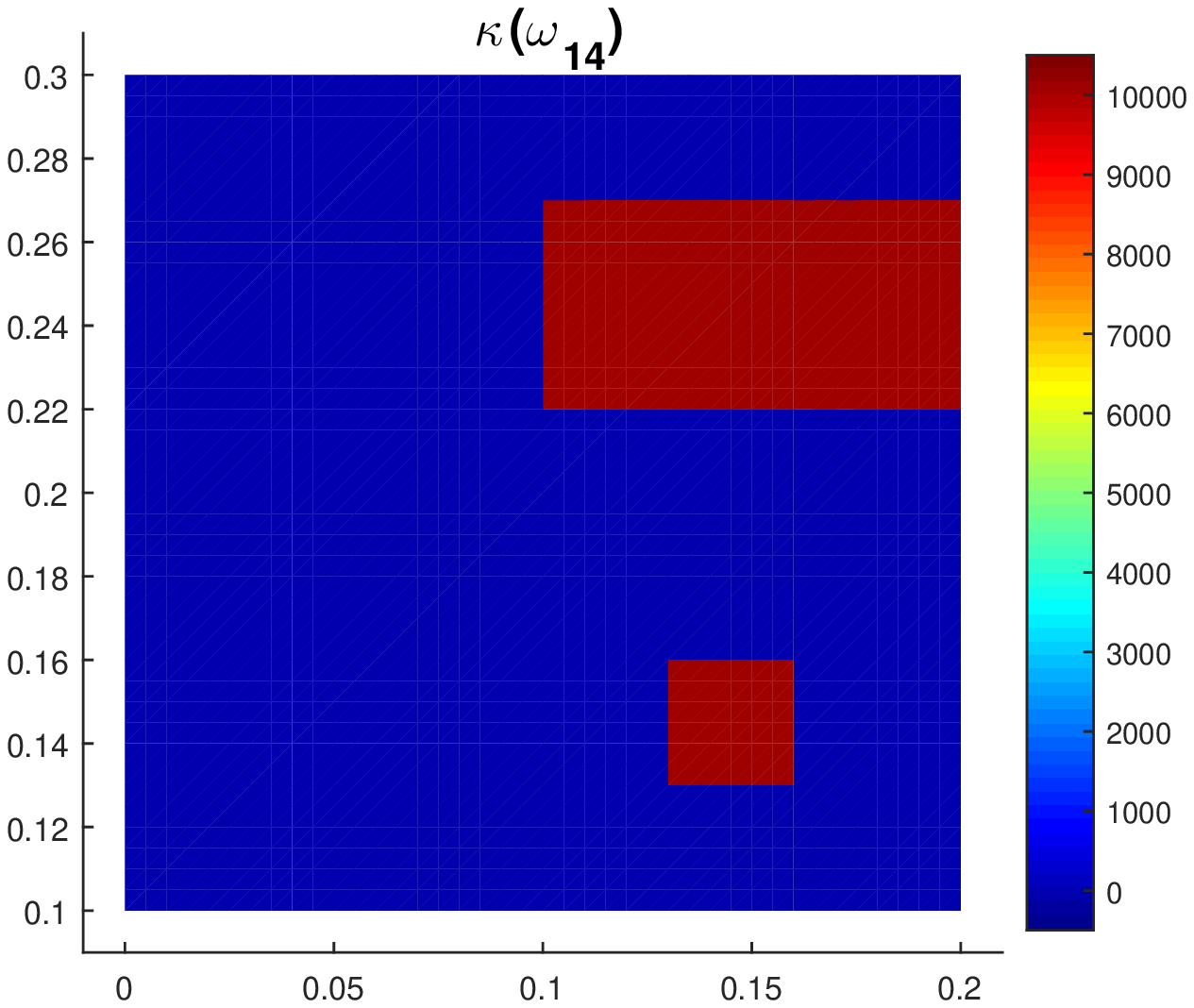}}
\subfigure[\tiny $\Sigma_{n}\|R_{25}^{n}\|_{Q^*}^{2}=1.1923E-07$.]{\includegraphics[width=1.6in, height=1.7in,angle=0]{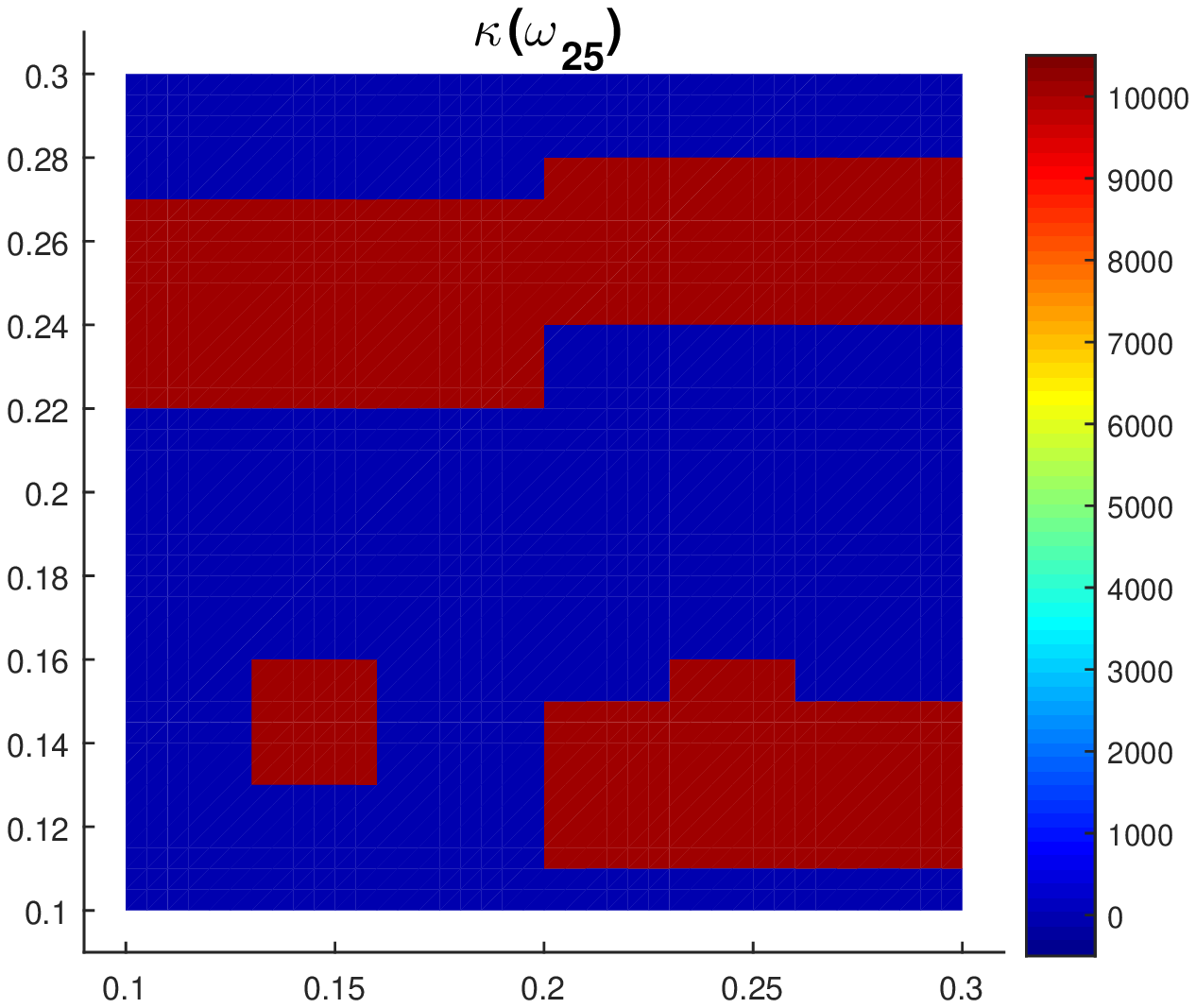}}
\subfigure[\tiny $\Sigma_{n}\|R_{39}^{n}\|_{Q^*}^{2}=1.2293E-05$.]{\includegraphics[width=1.6in, height=1.7in,angle=0]{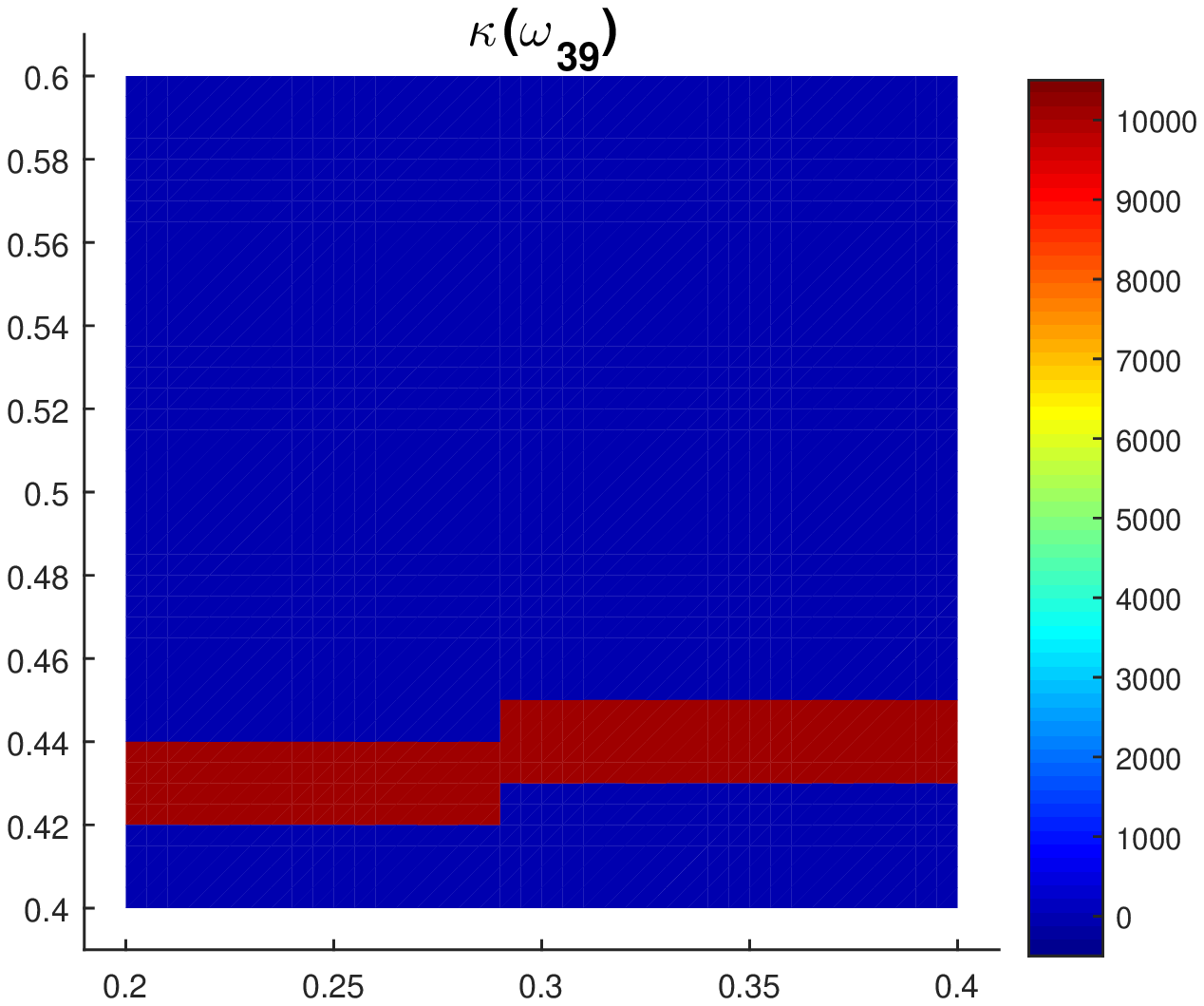}}
\caption{Some  local permeability fields ($N_{c}=121, H=1/10$).}
\label{cofNx4}
\end{figure}

\begin{figure}[htbp]
\centering
\subfigure[\tiny $\Sigma_{n}\|R_{68}^{n}\|_{Q^*}^{2}=4.9114E-09$ .]{\includegraphics[width=1.6in, height=1.7in]{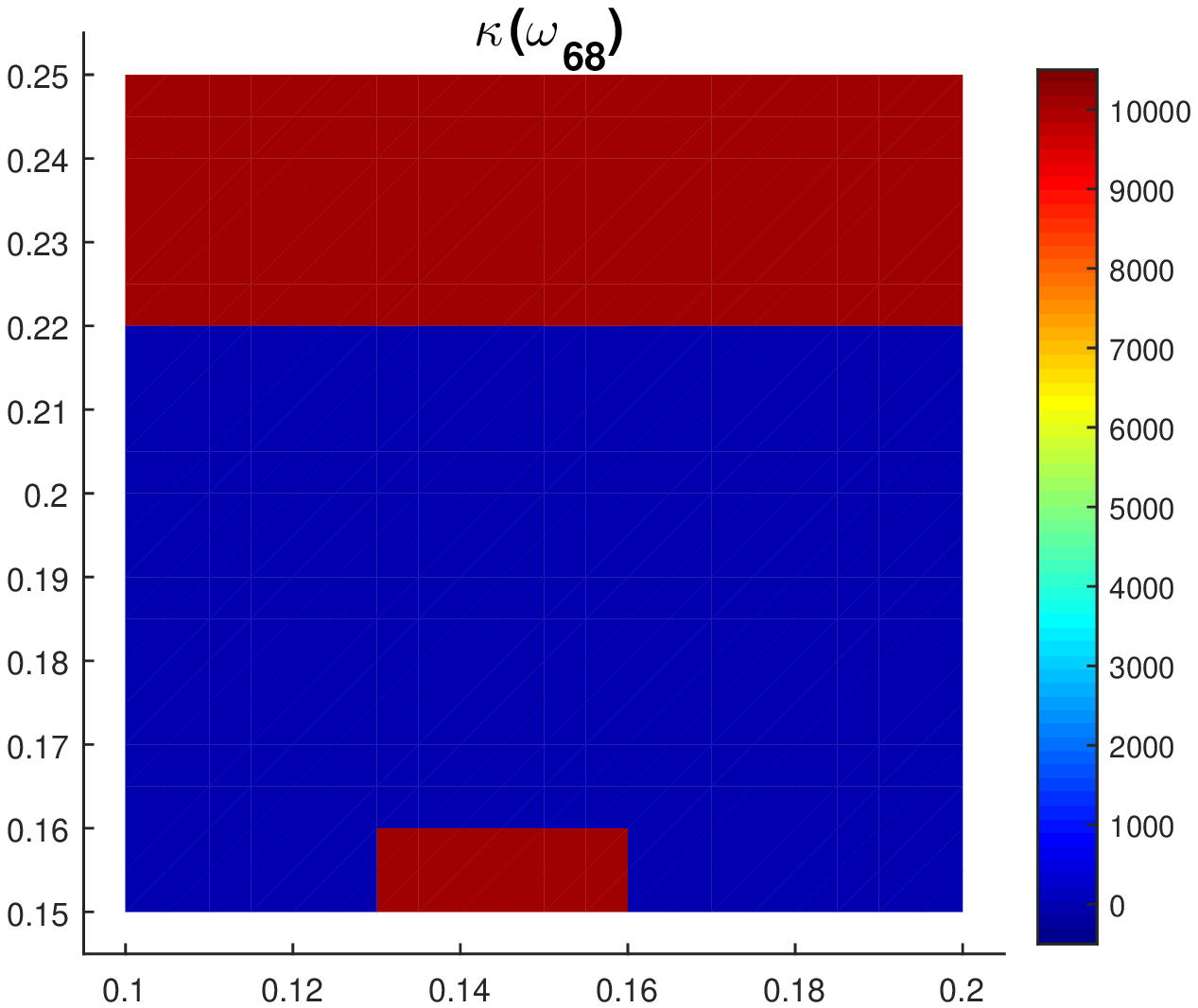}}
\subfigure[\tiny $\Sigma_{n}\|R_{71}^{n}\|_{Q^*}^{2}=1.4072E-08$.]{\includegraphics[width=1.6in, height=1.7in]{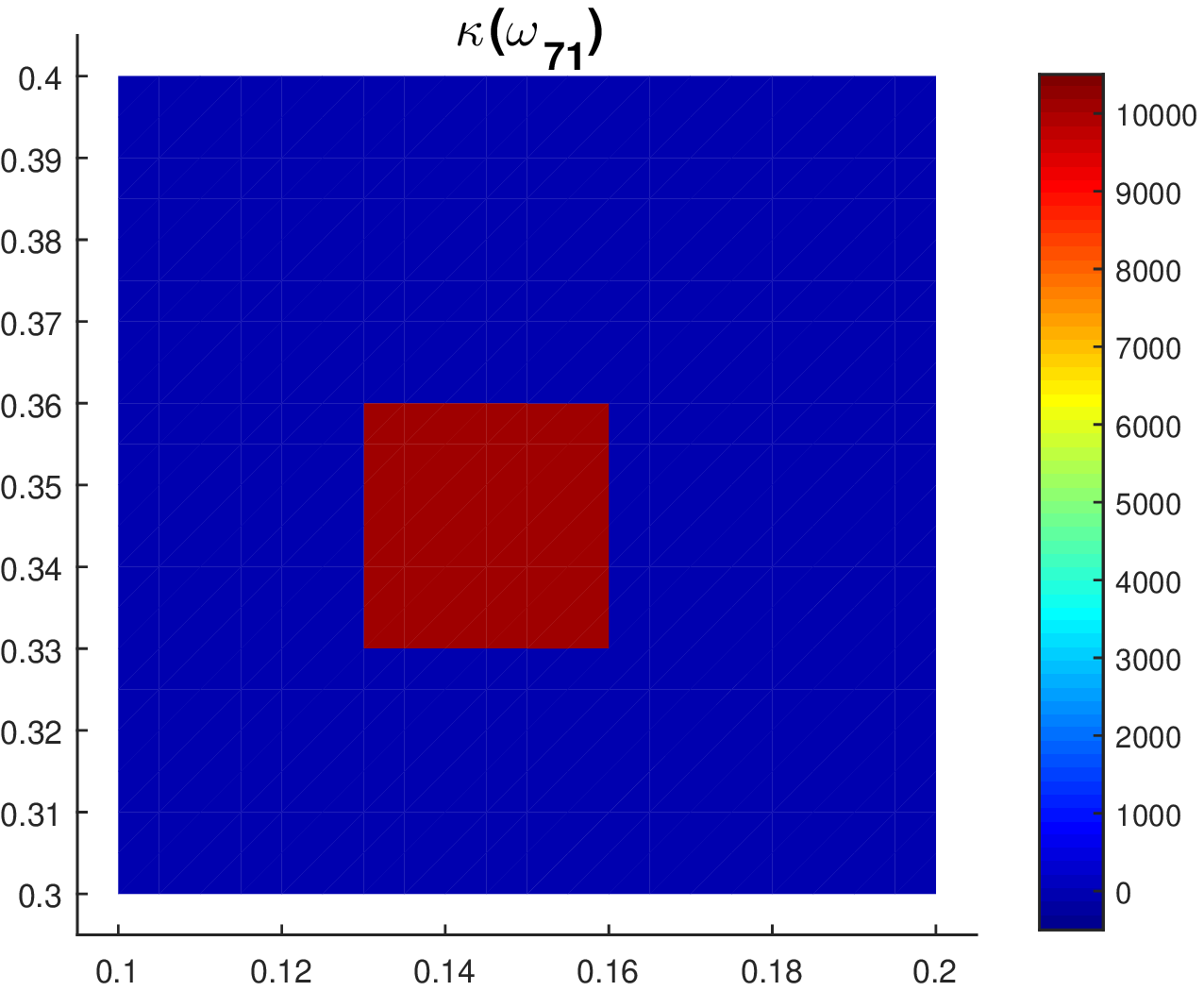}}
\subfigure[\tiny $\Sigma_{n}\|R_{246}^{n}\|_{Q^*}^{2}=1.0134E-07$.]{\includegraphics[width=1.6in, height=1.7in]{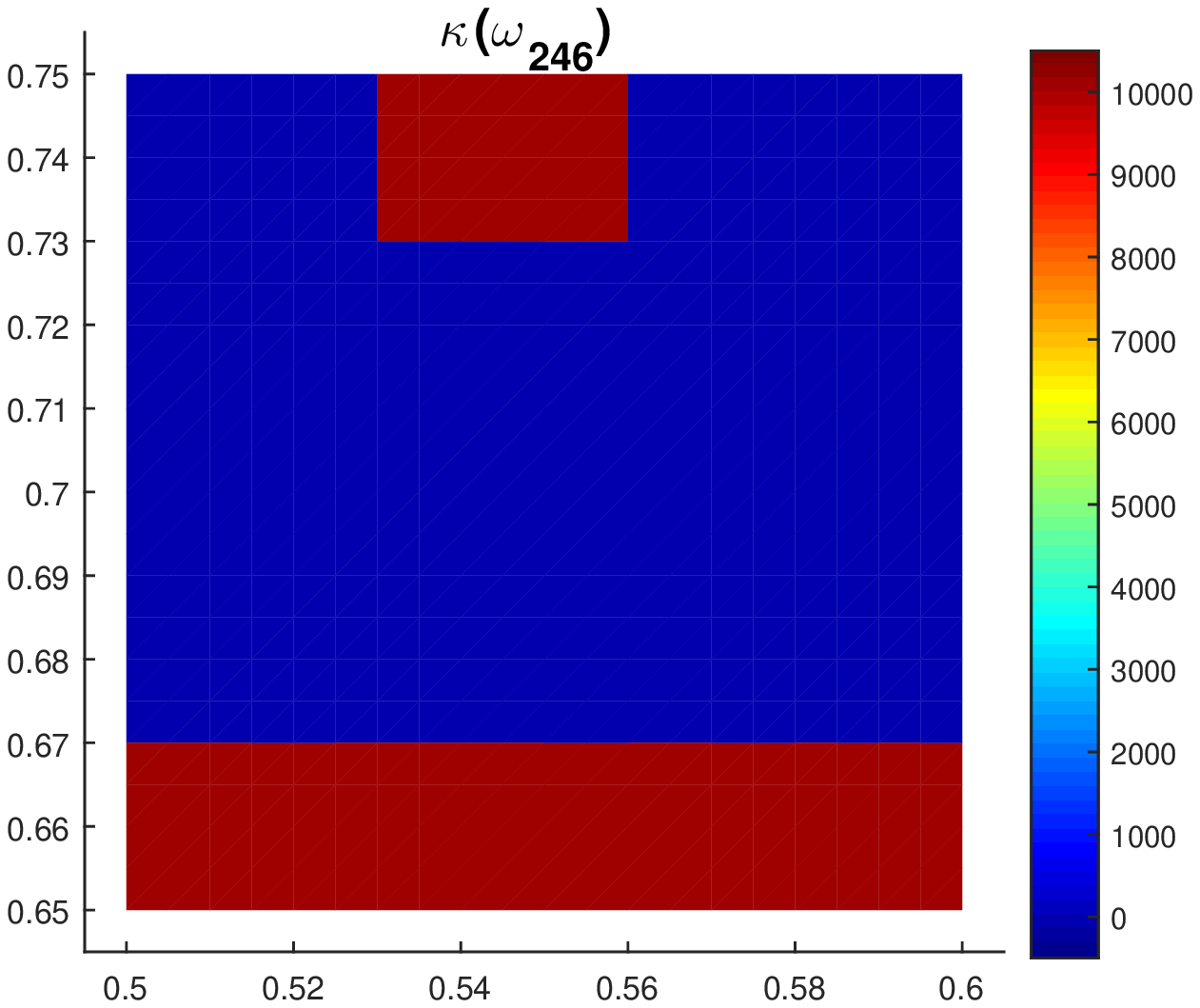}}
\caption{Some local permeability fields ($N_{c}=441, H=1/20$).}
\label{cofNx8}
\end{figure}


\section{Conclusion}
In this paper, we have presented  CEM-GMsFEM for solving the parabolic problems in multiscale porous  media. The analysis   showed that  the convergence depends on the coarse mesh size and the decay of eigenvalues of local spectral problems. For CEM-GMsFEM, the first step is to construct the auxiliary space through solving  local spectral  problems. The second step is to  solve a constraint energy minimizing problem to construct the multiscale basis functions in the oversampling local regions.  In this paper, we have constructed the elliptic projection in the space spanned  by  CEM-GMSFEM basis functions for convergence analysis.  Then we showed  the convergence of the semidiscrete formulation. The method can achieve the  second order convergence rate in the $L^{2}$ norm and first order convergence rate in the energy norm with respect to coarse mesh size. We also  derived the convergence and stability of the full discrete formulation.  A posteriori error bound was  given and the numerical results showed that it is a proper  upper bound.
 Some numerical results have been presented to  confirm the  theoretical results.



\end{document}